\documentclass{article}
\usepackage[top=2cm, bottom=2cm, left=2cm, right=2cm]{geometry}

\usepackage[alphabetic,lite]{amsrefs}                                          
\usepackage{amssymb}                                                           
\usepackage{amsthm}                                                            
\usepackage{appendix}                                                          
\usepackage{array}                                                             
\usepackage{bm}                                                                
\usepackage{colonequals}                                                       
\usepackage{color}                                                             
\usepackage{enumitem}                                                          
\usepackage{graphicx}                                                          
\usepackage{indentfirst}                                                       
\usepackage{mathrsfs}                                                          
\usepackage{mathtools}                                                         
\usepackage{moreenum}                                                          
\usepackage{stmaryrd}                                                          
\usepackage{verbatim}                                                          
\usepackage[all,cmtip]{xy}                                                     

\definecolor{tianred}{rgb}{0.79, 0.17, 0.57}                                   
\definecolor{tianblue}{rgb}{0.0, 0.22, 0.66}                                   
\definecolor{tianpink}{rgb}{0.88, 0.56, 0.59}                                  
\definecolor{tiangreen}{rgb}{0.24, 0.82, 0.44}                                 
\usepackage[colorlinks,
            linkcolor=tianred,
            anchorcolor=tianblue,
            citecolor=tiangreen,
            urlcolor=tianpink]{hyperref}                                       
\usepackage{cleveref}                                                          

\usepackage[OT2,T1]{fontenc}
\DeclareSymbolFont{cyrletters}{OT2}{wncyr}{m}{n}
\DeclareMathSymbol{\RBe}{\mathalpha}{cyrletters}{"42}                          
\DeclareMathSymbol{\Che}{\mathalpha}{cyrletters}{"51}                          
\DeclareMathSymbol{\Sha}{\mathalpha}{cyrletters}{"58}                          

\makeatletter
\DeclareRobustCommand\widecheck[1]{{\mathpalette\@widecheck{#1}}}
\def\@widecheck#1#2{%
    \setbox\z@\hbox{\m@th$#1#2$}%
    \setbox\tw@\hbox{\m@th$#1%
       \widehat{%
          \vrule\@width\z@\@height\ht\z@
          \vrule\@height\z@\@width\wd\z@}$}%
    \dp\tw@-\ht\z@
    \@tempdima\ht\z@ \advance\@tempdima2\ht\tw@ \divide\@tempdima\thr@@
    \setbox\tw@\hbox{%
       \raise\@tempdima\hbox{\scalebox{1}[-1]{\lower\@tempdima\box
\tw@}}}%
    {\ooalign{\box\tw@ \cr \box\z@}}}
\makeatother

\theoremstyle{plain}      \newtheorem{thm}{Theorem}[section]                   
\theoremstyle{plain}                        
\theoremstyle{plain}      \newtheorem{lem}[thm]{Lemma}                         
\theoremstyle{plain}                             
\theoremstyle{plain}      \newtheorem{cor}[thm]{Corollary}                     
\theoremstyle{plain}                         
\theoremstyle{plain}      \newtheorem{prop}[thm]{Proposition}                  
\theoremstyle{plain}      \newtheorem{conjecture}[thm]{Conjecture}             
\theoremstyle{definition} \newtheorem{rmk}[thm]{Remark}                        
\theoremstyle{definition}                      
\theoremstyle{definition} \newtheorem{df}[thm]{Definition}                     
\theoremstyle{definition}                   
\theoremstyle{definition} \newtheorem{eg}[thm]{Example}                        
\theoremstyle{definition}                        
\theoremstyle{definition}                        
\theoremstyle{definition}                      
\theoremstyle{definition}                    
\theoremstyle{definition}                  
\theoremstyle{definition}                        
\theoremstyle{definition}                       
\theoremstyle{definition}                    
\theoremstyle{definition}                  
\theoremstyle{definition}          
\theoremstyle{definition}              
\theoremstyle{definition}                  
\theoremstyle{definition} \newtheorem{prop-df}[thm]{Proposition-Definition}    

\theoremstyle{definition}
\newtheorem*{construction*}{Construction}                                      
\newtheorem*{conjecture*}{Conjecture}                                          
\newtheorem*{hypothesis*}{Hypothesis}                                          
\newtheorem*{convention*}{Convention}                                          
\newtheorem*{notation*}{Notation}                                              
\newtheorem*{summary*}{Summary}                                                
\newtheorem*{qt*}{Question}                                                    
\newtheorem*{rmk*}{Remark}                                                     
\newtheorem*{fact*}{Fact}                                                      
\newtheorem*{lizi*}{Example}                                                   
\newtheorem*{df*}{Definition}                                                  
\theoremstyle{plain}
\newtheorem*{thm*}{Theorem}                                                    

\crefname{thm}{Theorem}{Theorems}                                              %
\crefname{thme}{Th\'eo\`eme}{Th\'eo\`emes}
\crefname{lem}{Lemma}{Lemmas}
\crefname{lemme}{Lemme}{Lemmes}
\crefname{eg}{Example}{Examples}
\crefname{ege}{Exemple}{Exemples}
\crefname{rmk}{Remark}{Remarks}
\crefname{rmke}{Remarque}{Remarques}
\crefname{cor}{Corollary}{Corollaries}
\crefname{core}{Corollaire}{Corollaires}
\crefname{df}{Definition}{Definitions}
\crefname{dfe}{D\'efinition}{D\'efinitions}
\crefname{question}{Question}{Questions}
\crefname{prop}{Proposition}{Propositions}
\crefname{conjecture}{Conjecture}{Conjectures}

\newcommand{\benum}{\begin{enumerate}[label={{\upshape(\alph*)}}]}             
\newcommand{\benuma}{\begin{enumerate}[label={{\upshape(\arabic*)}}]}          
\newcommand{\benumr}{\begin{enumerate}[label={{\upshape(\roman*)}}]}           
\newcommand{\eenum}{\end{enumerate}}
\newcommand{\bconj}{\begin{conjecture}}
\newcommand{\econj}{\end{conjecture}}
\newcommand{\bconjnn}{\begin{conjecture*}}
\newcommand{\econjnn}{\end{conjecture*}}
\newcommand{\begs}{\begin{eg}\hfill\benuma}                                    
\newcommand{\eegs}{\eenum\end{eg}}                                             
\newcommand{\brmks}{\begin{rmk}\hfill\benuma}                                  
\newcommand{\ermks}{\eenum\end{rmk}}                                           
\newcommand{\bitem}{\begin{itemize}}                                           
\newcommand{\eitem}{\end{itemize}}                                             
\newcommand{\be}{\begin{equation}}                                             
\newcommand{\ee}{\end{equation}}                                               
\newcommand{\benn}{\begin{equation*}}                                          
\newcommand{\eenn}{\end{equation*}}                                            
\newcommand{\bqt}{\begin{qt*}\rm}                                              
\newcommand{\eqt}{\end{qt*}}                                                   
\newcommand{\bqtr}{\begin{qt*}\rm\coLR}                                        
\newcommand{\eqtr}{\end{qt*}}                                                  
\newcommand{\beac}{\begin{equation}\begin{array}{c}}                           
\newcommand{\eeac}{\end{array}\end{equation}}                                  
\newcommand{\beqn}{\begin{eqnarray*}}
\newcommand{\eeqn}{\end{eqnarray*}}
\newcommand{\bdf}{\begin{df}}
\newcommand{\bdfhf}{\begin{df}\hfill}
\newcommand{\edf}{\end{df}}
\newcommand{\brmk}{\begin{rmk}}
\newcommand{\brmkhf}{\begin{rmk}\hfill}
\newcommand{\ermk}{\end{rmk}}

\newcommand{\BI}{\mathbf{I}}   
   \newcommand{\BL}{\mathbf{L}}

\newcommand{\CA}{\mathcal{A}}  
\newcommand{\CC}{\mathcal{C}}  
  \newcommand{\CF}{\mathcal{F}}

\newcommand{\CM}{\mathcal{M}}  
\newcommand{\CO}{\mathcal{O}}  \newcommand{\CP}{\mathcal{P}}
  
  \newcommand{\CT}{\mathcal{T}}

  \newcommand{\BBD}{\mathbb{D}}
  
\newcommand{\BBG}{\mathbb{G}}  \newcommand{\BBH}{\mathbb{H}}

  \newcommand{\BBP}{\mathbb{P}}

  \newcommand{\rmD}{\mathrm{D}}

\newcommand{\rmI}{\mathrm{I}}

\newcommand{\olK}{\overline{K}}

\newcommand{\G}{\mathbb{G}}                                                    
\renewcommand{\P}{\mathbb{P}}                                                  
\newcommand{\Q}{\mathbb{Q}}                                                    
\newcommand{\Z}{\mathbb{Z}}                                                    
\newcommand{\QZ}{\mathbb{Q}/\mathbb{Z}}                                        
\newcommand{\al}{\alpha}                                                       

\newcommand{\ra}{\rightarrow}                                                  
\newcommand{\stra}[1]{\stackrel{#1}{\ra}}

\newcommand{\coLR}{\textcolor[rgb]{1.00,0,0}}                                  

\newcommand{\ol}{\overline}                                                    
\newcommand{\wt}{\widetilde}                                                   
\newcommand{\wh}{\widehat}                                                     
\newcommand{\wc}{\widecheck}                                                   

\newcommand{\ce}{\colonequals}                                                 
\newcommand{\uu}{^{\times}}                                                    
\newcommand{\ui}{^{-1}}                                                        
\newcommand{\uun}{^{(1)}}                                                      

\newcommand{\xm}{\xymatrix}

\newcommand{\drl}{\varinjlim}                                                  
\newcommand{\prl}{\varprojlim}                                                 

\DeclareMathOperator{\Hom}{Hom}                                                
\renewcommand{\hom}{\Hom}                                                      
\DeclareMathOperator{\ext}{Ext}                                                
\newcommand{\tstprodlim}{\tst\prod\limits}                                     
\newcommand{\tstopslim}{\tst\bigoplus\limits}                                  

\DeclareMathOperator{\Div}{Div}                                                
\DeclareMathOperator{\gal}{Gal}                                                
\newcommand{\pialg}[1]{\pi^{\alg}_1({\ol{#1}})}                                

\DeclareMathOperator{\Image}{Im}                                               
\renewcommand{\Im}{\Image}                                                     
\DeclareMathOperator{\Ker}{Ker}                                                
\renewcommand{\ker}{\Ker}                                                      
\DeclareMathOperator{\cok}{Coker}                                              
\DeclareMathOperator{\coker}{coker}		                                       

\DeclareMathOperator{\e}{Spec}                                                 

\newcommand{\ab}{\mathrm{ab}}                                                  
\newcommand{\alg}{\mathrm{alg}}                                                
\newcommand{\tors}{\mathrm{tors}}                                              
\newcommand{\cts}{\mathrm{cont}}                                               

\newcommand{\sconn}{\mathrm{sc}}                                               

\newcommand{\nr}{{\mathrm{nr}}}                                                
\newcommand{\sh}{\mathrm{sh}}                                                  

\DeclareMathOperator{\ots}{\otimes}                                            
\DeclareMathOperator{\dtp}{\otimes^{\BL}}                                      

\newcommand{\gm}{\BBG_m}                                                       
\newcommand{\Gss}{G^{\mathrm{ss}}}                                             
\newcommand{\Gsc}{G^{\mathrm{sc}}}                                             
\newcommand{\Tss}{T^{\mathrm{ss}}}                                             
\newcommand{\Tsc}{T^{\mathrm{sc}}}                                             
\newcommand{\Gtor}{G^{\mathrm{tor}}}                                           
\newcommand{\whT}{\wh{T}}                                                      %
\newcommand{\munots}[1]{\mu_n^{\ots{#1}}}                                      

\newcommand{\cmdm}{commutative diagram~}

\newcommand{\distri}{distinguished triangle~}
\newcommand{\distris}{distinguished triangles~}

\newcommand{\finiexp}{finite exponent~}

\newcommand{\hchlg}{hypercohomology~}

\newcommand{\inpart}{In particular,~}

\newcommand{\neos}{non-empty open subset~}
\newcommand{\neosbk}{non-empty open subset}

\newcommand{\ttes}{exact sequence~}
\newcommand{\ttess}{exact sequences~}
\newcommand{\ses}{short exact sequence~}

\newcommand{\wa}{weak approximation~}

\newcommand{\wrt}{with respect to~}

\newcommand{\ppair}{perfect pairing~}

\newcommand{\hzx}{functoriality~}
\newcommand{\hzxbk}{functoriality}
\newcommand{\ttai}{homomorphism~}
\newcommand{\ttais}{homomorphisms~}

\newcommand{\maxtorus}{maximal torus~}

\newcommand{\Kumseq}{Kummer sequence~}
\newcommand{\Kumseqs}{Kummer sequences~}

\newcommand{\ArtVer}{Artin--Verdier~}                                          
\newcommand{\TateSha}{Tate--Shafarevich~}                                      

\newcommand{\itm}{\item}
\newcommand{\tst}{\textstyle}

\newcommand{\pf}{\proof}
\newcommand{\pff}{\proof\hfill}
\newcommand{\vem}{\vspace{1em}}

\begin{document}
\title{\textbf{Poitou--Tate sequence for complex of tori over $p$-adic function fields}}

\author{Yisheng TIAN}

\maketitle

\begin{abstract}
We complete the picture of local and global arithmetic duality theorems for short complexes of finite Galois modules and tori over $p$-adic function fields.
In view of the duality theorems, we deduce a $12$-term Poitou--Tate exact sequence which relates global Galois cohomology groups to restricted topological products of local Galois cohomology groups.
\end{abstract}


\section*{Introduction}
Recently, there have been several developments concerning the arithmetic of linear algebraic groups over fields of arithmetic type having cohomological dimensional strictly larger than $2$
(for example, the function field $K$ of a smooth projective geometrically integral curve defined over some finite extension of $\Q_p$).
In \cites{CTPS12}, Colliot-Th\'el\`ene, Parimala and Suresh investigated the Hasse principle for certain varieties over such $K$.
Later, Hu \cite{Hu14} obtained results on the Hasse principle for simply connected groups
while Harari and Szamuely \cite{HSz16} found cohomological obstructions to the Hasse principle for quasi-split reductive groups
over $p$-adic function fields.
On the other hand, Harari, Scheiderer and Szamuely \cites{HSSz15, HSz16} studied systematically obstructions to the Hasse principle and \wa for tori over $p$-adic function fields.
All these developments motivate investigations concerning the arithmetic of reductive linear algebraic groups over such fields $K$.

This paper is some sort of complement to the preprint \cite{cat19WA} where the author established some arithmetic duality results and obtained obstructions to \wa for connected reductive algebraic groups over $K$.
In the present paper, we give a full picture of arithmetic duality results for a short complex of tori and deduce a $12$-term Poitou--Tate style \ttes
(in particular, we obtain such an \ttes for groups of multiplicative type over $K$).

Let us state the main results of this article.
Let $X$ be a smooth projective geometrically integral curve over a $p$-adic field $k$ and
let $K=k(X)$ be the function field of $X$.
Suppose $\rho:T_1\to T_2$ is a morphism of $K$-tori and
let $C=[T_1\to T_2]$ be the associated complex concentrated in degree $-1$ and $0$.
Let $T_1'$ and $T_2'$ be the respective dual torus of $T_1$ and $T_2$,
and put $C'=[T_2'\to T_1']$ for the dual complex concentrated in degree $-1$ and $0$.
We put $\Sha^i(C)\ce \ker\big(\BBH^i(K,C)\to \prod_{v\in X\uun}\BBH^i(K_v,C)\big)$ to be the \TateSha group of $C$,
where $X^{(1)}$ denotes the set of closed points on $X$.
The first main result is

\begin{thm*}
There is a perfect and functorial pairing of finite groups for $0\le i\le 2:$
\[\Sha^i(C)\times\Sha^{2-i}(C')\to\QZ.\]
\end{thm*}

So far, the theorem was known when $i=1$ by \cite{cat19WA}*{Theorem 1.17}.
Thus the above result provides a generalization and complement to \emph{loc. cit.} into all possible degrees
(see \cref{remark: on degrees III Sha C} for details).
In the context of higher dimensional local fields,
similar results of Izquierdo \cite{Diego-these}*{pp.~80, Th\'eor\`eme 4.17} are highly relevant
where he obtained perfect pairings between quotients of \TateSha groups by their maximal divisible subgroups.
In the case of number fields,
Demarche \cite{Dem11} established perfect pairings
$\Sha^i(C)\times\Sha^{2-i}(\wh{C})\to \QZ$
of finite groups when either $\ker\rho$ is finite or $\rho$ is surjective.
Here $\wh{C}=[\whT_2\to \whT_1]$ with $\whT_i$ the respective the module of characters associated to $T_i$.
Actually, we would like to consider the complex $C'=[T_2'\to T_1']$ instead of $\wh{C}=[\whT_2\to \whT_1]$ for the following reasons:

\bitem
\item
If we consider a connected reductive group $G$ and
the universal covering $\Gsc\to \Gss$ of its derived subgroup,
we may obtain an associated short complex $C=[\Tsc\to T]$
where $T$ is a maximal torus of $G$ and $\Tsc$ is the inverse image of $T$ in $\Gsc$.
Then \cite{cat19WA}*{Theorem 2.2(2)} tells us that the kernel of the map
$\Sha^1_{\omega}(C')^D\to \Sha^1(C)$
(which is induced by the global duality $\Sha^1(C)\simeq\Sha^1(C')^D$)
provides a defect of \wa for $G$.
Here $\Sha^1_{\omega}(C')$ denotes the subgroup of elements in $\BBH^1(K,C')$ that are zero in all but finitely many $\BBH^1(K_v,C')$.

\item
Let $M$ be a group of multiplicative type over $K$
and embed it into a \ses $0\to M\to T_1\to T_2\to 0$ with $T_i$ being $K$-tori.
Thus $M[1]$ is quasi-isomorphic to $C=[T_1\to T_2]$.
Actually it is too greedy to expect that the dual of $M$ still lies in the category of algebraic $K$-groups.
The natural "dual" of $M$ should be $C'=[T_2'\to T_1']$ because it does not lose the information on the torsion part of $\wh{M}$. The same phenomenon already arose when one tries to handle the dual of a semi-abelian variety.
As pointed out in \cite{HSz05}, the dual of a semi-abelian variety is a so-called $1$-motive.
\eitem

In order to connected pieces in the Poitou--Tate sequence (\ref{diagram: 12-term PT for short complex}) below,
we shall need an auxiliary global duality result.
For a closed point $v\in X\uun$, we let $\CO_v$ be the ring of integers in $K_v$.
We fix a non-empty open subset $X_0$ of $X$ such that $T_1$ and $T_2$ extends to $X_0$-tori $\CT_1$ and $\CT_2$ respectively,
and we put $\CC=[\CT_1\to \CT_2]$.
Let $\BBP^i(K,C)$ be the restricted topological product of $\BBH^i(K_v,C)$ \wrt the subgroups $\BBH^i(\CO_v,\CC)$
(we will show that $\BBH^i(\CO_v,\CC)$ is a subgroup of $\BBH^i(K_v,C)$ in the sequel).
For an abelian group $A$, we denote $A_{\wedge}=\prl_n A/n$.
Let $\Sha^0_{\wedge}(C)\ce \ker\big(\BBH^0(K,C)_{\wedge}\to \BBP^0(K,C)_{\wedge}\big)$.
We shall see that $\Sha^0_{\wedge}(C)\simeq\Sha^0(C)$ if $\ker\rho$ is finite in view of the following result.

\begin{thm*}
Suppose $\ker\rho$ is finite.
Then there is a perfect pairing of finite groups
\[\Sha^0_{\wedge}(C)\times\Sha^{2}(C')\to\QZ.\]
\end{thm*}

The proof of the theorem is analogous to that of \cite{Dem11}*{Proposition 5.10} in the number field context.
This result is more complicated because we have to handle the inverse limits $\BBH^0(K,C)_{\wedge}$ and $\BBP^0(K,C)_{\wedge}$,
and so it is not an immediate consequence of Artin--Verdier duality and local duality.
The idea is to describe $\Sha^0_{\wedge}(C)$ and $\Sha^{2}(C')$ by various limits.
Thus a crucial problem is to define respective transition maps and now the finiteness of $\ker\rho$ plays a role.
Finally, note that if $\coker\rho$ is trivial, then the kernel of $\rho':T_2'\to T_1'$ is finite.
Thus we obtain a perfect pairing $\Sha^{2}(C)\times\Sha^0_{\wedge}(C')\to\QZ$ of finite groups as well.

\vem
The classical Poitou--Tate sequence is a $9$-term \ttes which relates
global Galois cohomology with restricted ramification of
a finite Galois module over a global field (for example, see \cite{Harari17}*{Th\'eor\`eme 17.13} and also \cite{CK15-PT} for a generalization).
In \cite{HSSz15}, Harari, Scheiderer and Szamuely constructed a $12$-term Poitou--Tate style \ttes for finite Galois modules and a $9$-term one for tori in the $p$-adic function field context.
Now we arrive at the main result of the present article:

\begin{thm*}
Suppose either $\ker\rho$ is finite or $\cok\rho$ is trivial.
Then there is an \ttes of topological abelian groups
\beac\label{diagram: 12-term PT for short complex}
\xm@C=20pt @R=12pt{
0\ar[r] & \BBH\ui(K,C)_{\wedge}\ar[r] & \BBP\ui(K,C)_{\wedge}\ar[r] & \BBH^2(K,C')^D
\ar@{->} `r/5pt[d] `/8pt[l] `^dl/8pt[lll] `^r/8pt[dl][dll]\\
& \BBH^0(K,C)_{\wedge}\ar[r] & \BBP^0(K,C)_{\wedge}\ar[r] & \BBH^1(K,C')^D
\ar@{->} `r/5pt[d] `/8pt[l] `^dl/8pt[lll] `^r/8pt[dl] [dll]\\
& \BBH^1(K,C)\ar[r] & \BBP^1(K,C)_{\tors}\ar[r] & \big(\BBH^0(K,C')_{\wedge}\big)^D
\ar@{->} `r/5pt[d] `/8pt[l] `^dl/8pt[lll] `^r/9pt[dl] [dll]\\
& \BBH^2(K,C)\ar[r] & \BBP^2(K,C)_{\tors}\ar[r] & \big(\BBH\ui(K,C')_{\wedge}\big)^D\ar[r] & 0
}
\eeac
where $\BBP^i(K,C)_{\tors}$ denotes the torsion subgroup of the group $\BBP^i(K,C)$ for $i=1,2$.
\end{thm*}

Actually, the exactness of the first and the last row in diagram (\ref{diagram: 12-term PT for short complex}) hold without any assumption on $C$.
The finiteness of the kernel $\ker\rho$ or the surjectivity of $\rho$ plays an essential role in the proof of the exactness of the middle two rows.
Moreover, as we have seen the finiteness of $\ker\rho$ provides perfect pairings between \TateSha groups which enable us to connect the desired exact sequence.
Finally, let us have a glimpse at consequences
(see \cref{example: PT seq for tori mult type reductive group} for details)
of the Poitou--Tate sequence given above.

\bitem
\item
If we take $C=[0\to P]$ to be a single torus, then we obtain \cite{HSSz15}*{Theorem 2.9}.

\item
For a connected reductive group $G$ and a maximal torus $T$ in $G$,
let $\Tsc\subset \Gsc$ be as above.
Then $\ker(\Tsc\to T)$ is finite and we deduce a Poitou--Tate sequence
(see \cref{example: PT seq for tori mult type reductive group}(2))
for the reductive group $G$.

\item
If $\rho$ is surjective,
then the complex $C[-1]$ is quasi-isomorphic to $\ker\rho$ (which is a group of multiplicative type).
Thus we obtain Poitou--Tate style \ttess for groups of multiplicative type.
\eitem

Let us close the introduction by potential applications of various Poitou--Tate style exact sequences.
Harari and Izquierdo set forth how to use a Poitou--Tate sequence \cite{HI18}*{Th\'eor\`eme 4.6} to define a defect to strong approximation using a divisible quotient group in \cite{HI18}*{Th\'eor\`eme 4.8} when the base field is the function field of a smooth projective curve defined over an algebraically closed field of characteristic zero.
Besides, Demarche \cite{Dem11AF}*{Th\'eor\`eme 2.9} studied a defect to strong approximation with the help of another Poitou--Tate sequence.
Hopefully our Poitou--Tate sequence gives an interpretation of a defect to strong approximation for connected linear groups as well over $p$-adic function fields.

\vem
\textbf{Acknowledgements}.
I thank my advisor David Harari for many useful discussions and helpful comments.
I thank the EDMH doctoral program for support and Universit\'e Paris-Sud for excellent conditions for research.

\section*{Notations and conventions}

\textbf{Function fields}. Throughout this article, $K$ will be the function field of a smooth proper and geometrically integral curve $X$ over a $p$-adic field.
Note that each closed point $v\in X^{(1)}$ defines a discrete valuation of $K$.
We write $\CO_{X,v}$ for the local ring at $v$ and $\kappa(v)$ for its residue field.
Moreover, $K_v$ (resp. $K_v^h$) will be the completion (resp. Henselization) of $K$ with respect to $v$ and $\CO_v$ (resp. $\CO_{v}^h$) will be the ring of integers in $K_v$ (resp. $K_v^h$).

\textbf{Abelian groups}.
Let $A$ be an abelian group.
We shall denote by $_nA$ (resp. $A\{\ell\}$) for the $n$-torsion subgroup
(resp. $\ell$-primary subgroup with $\ell$ a prime) of $A$.
Moreover, let $A_{\tors}$ be the torsion subgroup of $A$, so $A_{\tors}=\drl_n {_n}A$ is the direct limit of $n$-torsion subgroups of $A$.
We write $A^{\wedge}$ for the profinite completion of $A$ (that is, the inverse limit of its finite quotients),
$A_{\wedge}\ce \prl_n A/nA$ and
$A^{(\ell)}\ce\prl_n A/\ell^n$ for the $\ell$-adic completion with $\ell$ a prime number.
A torsion abelian group $A$ is of cofinite type if $_nA$ is finite for each $n\ge 1$.
If $A$ is $\ell$-primary torsion of cofinite type, then $A/\Div A\simeq A^{(\ell)}$ where the former group is the quotient of $A$ by its maximal divisible subgroup.
For a topological abelian group $A$, we write $A^D\ce \hom_{\cts}(A,\QZ)$ for the group of continuous \ttais.

\textbf{Motivic complexes}.
Let $L$ be a field.
For a smooth $L$-variety $Y$, we denote the \'etale motivic complex over $Y$ by the complex of sheaves
$\Z(i)\ce z^i(-,\bullet)[-2i]$ on the small \'etale site of $Y$,
where $z^i(Y,\bullet)$ is Bloch's cycle complex \cite{Bloch86}.
For example, we have quasi-isomorphisms $\Z(0)\simeq\Z$ and $\Z(1)\simeq\G_m[-1]$ by \cite{Bloch86}*{Corollary 6.4}.
We write $A(i)\ce A\dtp \Z(i)$ for any abelian group $A$.
Finally, if $n$ is an integer invertible in $L$, then \cite{GL01:Bloch-Kato}*{Theorem 1.5} gives a quasi-isomorphism $\Z/n\Z(i)\simeq \mu_n^{\ots i}$ where $\mu_n$ is concentrated in degree $0$.
We shall write $\QZ(i)\ce \drl_n \mu_n^{\ots i}$ for the direct limit of the sheaves $\mu_n^{\ots i}$ for all $n\ge 1$.

\textbf{Tori and short complex of tori}.
Let $L$ be a field of characteristic zero and let $\ol{L}$ be a fixed algebraic closure of $L$.
We write $\wh{T}$ or $X^*(T)$ (resp. $\wc{T}$ or $X_*(T)$) for the character module (resp. cocharacter module) of a $L$-torus $T$.
These are finitely generated free abelian groups endowed with a $\gal(\ol{L}|L)$-action, and moreover $\wc{T}$ is the $\Z$-linear dual of $\wh{T}$.
The dual torus $T'$ of $T$ is the torus with character group $\wc{T}$, that is, $\wh{T'}=\wc{T}$.

Let $C=[T_1\stra{\rho}T_2]$ be a short complex of $L$-tori concentrated in degree $-1$ and $0$.
We always write $M=\ker\rho$, $T=\cok\rho$ and $C'=[T_2'\to T_1']$ (again it is concentrated in degree $-1$ and $0$).
Thus $M$ is a group of multiplicative type and $T$ is a torus.
Let $X_0\subset X$ be a non-empty open subset such that $T_1$ and $T_2$ extend to $X_0$-tori $\CT_1$ and $\CT_2$ respectively (in the sense of \cite{SGA3II}*{Expos{\'e} IX, D\'efinition 1.3}).
We similarly write $\CM=\ker(\CT_1\to \CT_2)$ and $\CT=\cok(\CT_1\to \CT_2)$ over $X_0$.
So over $X_0$ we obtain complexes $\CC=[\CT_1\to \CT_2]$ and $\CC'=[\CT_2'\to \CT_1']$.
Finally, $U$ will always be a suitably sufficiently small non-empty open subset of $X_0$.

For the short complex $\CC=[\CT_1\to \CT_2]$ over $X_0$, we put, for $n\ge 1$, $T_{\Z/n}(\CC)\ce H^0(\CC[-1]\dtp\Z/n)$ which is an fppf sheaf of abelian groups.
Recall \cite{Dem11}*{Lemme 2.3} that the sheaf $T_{\Z/n}(\CC)$ is represented by a finite group scheme of multiplicative type over $X_0$,
and it fits into a \distri ${_n}\CM[2]\to \CC\dtp\Z/n\to T_{\Z/n}(\CC)[1]\to {_n}\CM[3]$.

\textbf{Cohomologies}. Unless otherwise stated, all cohomologies are understood with respect to the \'etale topology.
Let $j_0:X_0\to X$ be the open immersion.
We denote $\BBH^i_c(X_0,\CC)\ce \BBH^i(X_0,j_{0!}\CC)$ for the compact support cohomology.

\textbf{Pairings}.
Let $L$ be either $K$ or $K_v$.
There is a canonical pairing $C\dtp C'\to \Z(2)[3]$
(see \cite{Diego-these}*{pp.~69, Lemme 4.3})
over $K$ which extends to a pairing $\CC\dtp\CC'\to \Z(2)[3]$ over $X_0$
(both pairings are in the bounded derived category of \'etale sheaves).
By \cite{HSz16}*{Lemma 1.1}, there are canonical pairings
$\BBH^i(U,\CC)\times\BBH^{2-i}_c(U,\CC')\to \QZ$ for any non-empty open subset $U\subset X_0$ and
$\BBH^i(L,C)\times \BBH^{1-i}(L,C')\to \QZ$.
If we consider further the pairing $\Z/n\dtp\Z/n\to \Z[1]$, we obtain a canonical pairing
$(C\dtp\Z/n)\dtp(C'\dtp\Z/n)\to \Z(2)[4]$ which induces canonical pairings
$\BBH^i(U,\CC\dtp\Z/n)\times\BBH^{1-i}_c(U,\CC'\dtp\Z/n)\to \QZ$ and
$\BBH^i(L,C\dtp\Z/n)\times \BBH^{-i}(L,C'\dtp\Z/n)\to \QZ$.

\textbf{Triangles}.
We shall use the following distinguished triangles frequently in the sequel.
By definition of $C$, there is a \distri $T_1\to T_2\to C\to T_1[1]$
and similarly ${_n}T_1\to {_n}T_2\to C\dtp\Z/n[-1]\to {_n}T_1[1]$.
By definition of $M$ and $T$, there is a \distri $M[1]\to C\to T\to M[2]$.
There is a \distri $C\to C\to C\dtp\Z/n\to C[1]$ obtained from the Kummer sequences for tori.
Finally, there is a \distri ${_n}M[2]\to C\dtp\Z/n\to T_{\Z/n}(C)[1]\to {_n}M[3]$ by \cite{Dem11}*{Lemme 2.3},
where $T_{\Z/n}(C)\ce H^0(C[-1]\dtp\Z/n)$.

\section{Preliminaries in finite level}
We first develop some arithmetic duality results and a $15$-term Poitou--Tate \ttes concerning the complexes $C\dtp\Z/n$ and $C'\dtp\Z/n$ for any $n\ge 1$.

\subsection{Local dualities}
\begin{prop}\label{duality: local-finite level}
There is a \ppair of finite groups for $i\in\Z$
\be\label{pairing: local-finite level}
\BBH^i(K_v,C\dtp\Z/n)\times\BBH^{-i}(K_v,C'\dtp\Z/n)\to \QZ.
\ee
\end{prop}
\begin{proof}
Recall \cite{HSz16}*{pp.~6, pairing (10)} that there is a \ppair of finite groups for $j\in\Z$
\[H^{j}(K_v,{_n}T)\times H^{3-j}(K_v,{_n}T')\to \QZ.\]
Therefore the \distris
${_n}T_1[1]\to {_n}T_2[1]\to C\dtp\Z/n\to {_n}T_1[2]$ and
${_n}T_2'[1]\to {_n}T_1'[1]\to C'\dtp\Z/n\to {_n}T_2'[2]$
yield an isomorphism $\BBH^i(K_v,C\dtp\Z/n)\simeq \BBH^{-i}(K_v,C'\dtp\Z/n)^D$ by d\'evissage.
\end{proof}

%
%

\begin{rmk}\label{remark: on degrees I}
To proceed, let us first briefly explain the degrees under consideration.
\benuma
\item
Let $P$ be a $K$-torus.
The groups $H^i(K,P)$ and $H^i(K_v,P)$ vanish for $i\ge 3$.
See \cite{SvH03}*{Corollary 4.10} for the former group and see \cite{HSz16}*{Remark 2.3} for the latter one.

\item
The group $H^i(K_v,C\dtp\Z/n)=0$ for $i\le -3$ or $i\ge 3$.
This is a direct consequence of d\'evissage thanks to the \distri $C\to C\to C\dtp\Z/n\to C[1]$.

\item
The groups $\BBH^i(\CO_v,\CC\dtp\Z/n)$ vanish for $i\ge 2$ or $i\le -3$.
Indeed, we consider over $\CO_v$ the \distri ${_n}\CT_1[1]\to {_n}\CT_2[1]\to \CC\dtp\Z/n\to {_n}\CT_1[2]$.
Therefore it will be sufficient to show $H^i(\CO_v,{_n}\CP)=0$ for $i\ge 3$ and $i\le -1$, and for any $\CO_v$-torus $\CP$ by d\'evissage.
Finally, $H^i(\CO_v,{_n}\CP)=H^i(\kappa(v),{_n}\CP)=0$ for $i\ge 3$ for cohomological dimension reasons
(the first identification follows from \cite{MilneADT}*{II, Proposition 1.1(b)}),
and $H^i(\CO_v,{_n}\CP)=0$ for $i\le -1$ by construction.
\eenum
\end{rmk}

\begin{lem}\label{lemma: inj of cohomology from Ov to Kv for cpx of tori in finite level}
For $-2\le i\le 1$, the homomorphism $\BBH^i(\CO_v,\CC\dtp\Z/n)\to \BBH^i(K_v,C\dtp\Z/n)$ induced by the inclusion $\CO_v\subset K_v$ is injective.
Thus we may view $\BBH^i(\CO_v,\CC\dtp\Z/n)$ as a subgroup of $\BBH^i(K_v,C\dtp\Z/n)$.
\end{lem}
\begin{proof}
Let $\olK_v$ be a separable closure of $K_v$ and let $K_v^{\nr}$ be the maximal unramified extension of $K_v$.
First note that
$\BBH^i(\CO_v,\CC\dtp\Z/n)\simeq \BBH^i(\kappa(v),\CC\dtp\Z/n)$
by \cite{MilneADT}*{II, Proposition 1.1(b)} and d\'evissage,
where the latter group is isomorphic to $\BBH^i(\gal(K_v^{\nr}|K_v),\CC\dtp\Z/n)$ by ramification theory.
Choose an extension of $v$ to $\olK_v$ and let $I_v$ be the corresponding inertia group.
But the \ses $1\to I_v\to \gal(\olK_v|K_v)\to \gal(K_v^{\nr}|K_v)\to 1$ admits a section \cite{SerreCG}*{II, Appendix},
consequently $H^i(K_v^{\nr}|K_v,\CC\dtp\Z/n)\to H^i(K_v,C\dtp\Z/n)$ admits a retraction, hence injective.
\end{proof}

We shall need the following additional result on respective annihilators of local dualities.

\begin{prop}\label{duality: annihilators of local pairing in finite level}
For $-2\le i\le 2$, the annihilator of $\BBH^i(\CO_v,\CC\dtp\Z/n)$ is $\BBH^{-i}(\CO_v,\CC'\dtp\Z/n)$ under the \ppair
\[
\BBH^i(K_v,C\dtp\Z/n)\times\BBH^{-i}(K_v,C'\dtp\Z/n)\to \QZ.
\]
\end{prop}
\pf
The \distri ${_n}\CT_1[1]\to {_n}\CT_2[1]\to \CC\dtp\Z/n\to {_n}\CT_1[2]$ over $\CO_v$ yields a \cmdm
\[
\xm@C=14pt{
H^{i+1}(\CO_v,{_n}\CT_1)\ar[r]\ar[d] & H^{i+1}(\CO_v,{_n}\CT_2)\ar[r]\ar[d] & \BBH^i(\CO_v,\CC\dtp\Z/n)\ar[r]\ar[d] &
H^{i+2}(\CO_v,{_n}\CT_1)\ar[r]\ar[d] & H^{i+2}(\CO_v,{_n}\CT_2)\ar[d] \\
H^{i+1}(K_v,{_n}T_1)\ar[r]\ar[d] & H^{i+1}(K_v,{_n}T_2)\ar[r]\ar[d] & \BBH^i(K_v,C\dtp\Z/n)\ar[r]\ar[d] &
H^{i+2}(K_v,{_n}T_1)\ar[r]\ar[d] & H^{i+2}(K_v,{_n}T_2)\ar[d] \\
H^{2-i}(\CO_v,{_n}\CT'_1)^D\ar[r] & H^{2-i}(\CO_v,{_n}\CT'_2)^D\ar[r] & \BBH^{-i}(\CO_v,\CC'\dtp\Z/n)^D\ar[r] &
H^{1-i}(\CO_v,{_n}\CT'_1)^D\ar[r] & H^{1-i}(\CO_v,{_n}\CT'_2)^D
}
\]
of finite groups with exact rows.
Moreover, all the columns except the middle one are exact by \cite{HSSz15}*{Proposition 1.2}.
For cohomological dimension reasons, for any $X_0$-torus $\CP$, the pairing
\[
\BBH^i(\CO_v,{_n}\CP)\times\BBH^{3-i}(\CO_v,{_n}\CP')\to \BBH^3(\CO_v,\QZ(2))\simeq\BBH^3(\kappa(v),\QZ(2))
\]
is trivial and hence the columns in the diagram are complexes.
\benuma
\item
$i=-2$.
In this case, we have $H^{i+2}(\CO_v,{_n}\CT_j)\simeq H^{i+2}(K_v,{_n}T_j)$ and $H^{i+1}(\CO_v,{_n}\CT_j)=H^{i+1}(K_v,{_n}T_j)=0$ for $j=1,2$.
Thus exactness of the middle column fulfills after a diagram chasing.
\item
$i=-1$.
The right two columns of the diagram above are exact, and $H^{i+1}(\CO_v,{_n}\CT_2)\simeq H^{i+1}(K_v,{_n}T_2)$ is an isomorphism.
A diagram chasing yields the exactness of the middle column.
\item
$i=0$.
Note that $H^{i+1}(K_v,{_n}T_1)\to H^{-i+2}(\CO_v,{_n}\CT'_1)^D$ is surjective
because $H^{-i+2}(\CO_v,{_n}\CT'_1)\to H^{-i+2}(K_v,{_n}T_1')$ is an inclusion (see \cite{HSSz15}*{Proposition 1.2}) of finite groups.
The exactness of the middle column follows from another diagram chasing.
\qed
\eenum

We denote by $\BBP^i(K,C\dtp\Z/n)\ce\prod'\BBH^i(K_v,C\dtp\Z/n)$ the restricted topological product of the finite discrete groups
$\BBH^i(K_v,C\dtp\Z/n)$ \wrt $\BBH^i(\CO_v,\CC\dtp\Z/n)$.
Note that the only non-trivial degrees are $-2\le i\le 2$ by \cref{remark: on degrees I}.
Since $\BBP^i(K,C\dtp\Z/n)$ is a direct limit of discrete groups, it is locally compact.
Moreover, since
$\BBP^{-2}(K,C\dtp\Z/n)=\prod_{v\in X^{(1)}}\BBH^{-2}(K_v,C\dtp\Z/n)$ and
$\BBP^{ 2}(K,C\dtp\Z/n)=\bigoplus_{v\in X^{(1)}}\BBH^2(K_v,C\dtp\Z/n)$
by \cref{remark: on degrees I},
$\BBP^{-2}(K,C\dtp\Z/n)$ is profinite,
and
$\BBP^{2}(K,C\dtp\Z/n)$ is discrete (it is a direct sum of finite groups).

Similarly, we let $\BBP^i(K,C)$ be the restricted topological product of the groups
$\BBH^i(K_v,C)$ \wrt the subgroups $\BBH^i(\CO_v,\CC)$ (see \cref{injectivity: cohohomology of O into K for -1 0 1 2} below) for $v\in X_0\uun$ and $-1\le i\le 2$.

\begin{cor}
For $-2\le i\le 2$, there are isomorphisms
$\BBP^i(K,C\dtp\Z/n)\simeq \BBP^{-i}(K,C'\dtp\Z/n)^D$
of locally compact topological groups.
\end{cor}
\begin{proof}
This is an immediate consequence of \cref{duality: local-finite level} and \cref{duality: annihilators of local pairing in finite level}.
\end{proof}

\subsection{Global dualities}
We begin with an \ArtVer style duality result which plays a role in the proof of the global duality
\[\Sha^i(C\dtp\Z/n)\times\Sha^{1-i}(C'\dtp\Z/n)\to \QZ\]
for $-1\le i\le 2$.
We quote the following proposition \cite{Diego-these}*{pp.~70, I.4.4} for convenience and completeness.

\begin{prop}[\ArtVer duality]\label{duality: AV-finite level}
Let $U\subset X_0$ be a \neosbk.
For $i\in\Z$, there is a \ppair between finite groups
\[\BBH^i(U,\CC\dtp\Z/n) \times \BBH^{1-i}_c(U,\CC'\dtp\Z/n) \to \QZ.\]
\end{prop}

Another input for the proof of global duality is the key \ttes (\ref{sequence: key ex-seq in deg 1 with finite coeff}) below.
We need the following results to insure its exactness.

\begin{lem}\label{lemma: Hen vs completion in finite level}
There are canonical isomorphisms $\BBH^i(K_v^h,C\dtp\Z/n)\simeq\BBH^i(K_v,C\dtp\Z/n)$ for all $i\ge -1$.
\end{lem}
\begin{proof}
Let $F$ be a finite \'etale commutative group scheme over $K$.
Note that $F$ is locally constant in the \'etale topology
 and that $K_v^h$ and $K_v$ have the same absolute Galois group,
therefore $H^i(K_v^h,F)\simeq H^i(K_v,F)$ for any $i\in \Z$.
Now the result follows thanks to the \distri ${_n}T_1\to {_n}T_2\to C\dtp\Z/n[-1]\to {_n}T_1[1]$ by d\'evissage.
\end{proof}

The following proposition is proved in \cite{cat19WA}*{Proposition 1.10}.
Since it frequently plays a role in the demonstrations later, we quote it here for convenience of the reader.

\begin{prop}\label{sequence: long seq cpt supp and 3 arrows lemma}
Let $U\subset X_0$ be a non-empty open subset.
Let $\CA$ be either $\CC$ or $\CC\dtp\Z/n$.
\benuma
\item Let $V\subset U$ be a further non-empty open subset. There is an exact sequence
$$\cdots\to \BBH^i_c(V,\CA)\to \BBH^i_c(U,\CA)\to \bigoplus_{v\in U\setminus V}\BBH^i(\kappa(v),i_v^*\CA)\to \BBH^{i+1}_c(V,\CA)\to\cdots$$
where $i_v:\e\kappa(v)\to U$ is the closed immersion.

\item There is an exact sequence of \hchlg groups
$$\cdots\to \BBH^i_c(U,\CA)\to \BBH^i(U,\CA)\to \bigoplus_{v\notin U}\BBH^i(K_v^h,A)\to \BBH^{i+1}_c(U,\CA)\to \cdots$$
where $K_v^h$ is the Henselization of $K$ with respect to the place $v$ and by abuse of notation we write $A$ for the pull-back of $\CA$ by the natural morphism $\e K_v^h\to U$.

\item There is an exact sequence for $i\ge 1$ if $\CA=\CC$, and for $i\ge -1$ if $\CA=\CC\dtp\Z/n$:
$$\cdots\to \BBH^i_c(U,\CA)\to \BBH^i(U,\CA)\to \bigoplus_{v\notin U}\BBH^i(K_v,A)\to \BBH^{i+1}_c(U,\CA)\to \cdots$$

\item \textup{(Three Arrows Lemma).} There is a \cmdm
\[
\xm{
\BBH^i_c(V,\CA)\ar[r]\ar[d] & \BBH^i_c(U,\CA)\ar[d]\\
\BBH^i(V,\CA) & \BBH^i(U,\CA).\ar[l]
}
\]
\eenum
\end{prop}

Put $\BBD^{i}_K(U,\CC\dtp\Z/n)\ce \Im\big(\BBH^i_c(U,\CC\dtp\Z/n)\to \BBH^i(K,C\dtp\Z/n)\big)$.
Now we arrive at the key \ttes for the proof of global duality between the respective \TateSha groups of $C\dtp\Z/n$ and $C'\dtp\Z/n$.

\begin{prop}\label{proposition: key ex-seq in deg -1 to 1 with finite coeff}
There is an \ttes for $-1\le i\le 1$
\be\label{sequence: key ex-seq in deg 1 with finite coeff}
\bigoplus_{v\in X^{(1)}}\BBH^{i}(K_v,C\dtp\Z/n)\to \BBH^{i+1}_c(U,\CC\dtp\Z/n)\to \BBD^{i+1}_K(U,\CC\dtp\Z/n)\to 0.
\ee
\end{prop}
\begin{proof}
There is a well-defined map $\bigoplus_{v\in X^{(1)}}\BBH^{i}(K_v,C\dtp\Z/n)\to \BBH^{i+1}_c(U,\CC\dtp\Z/n)$ in a similar way as \cite{HSz16}*{pp.~11}.
Let us recall the construction for the convenience of the readers.
Suppose $\al\in\bigoplus_{v\in X^{(1)}}\BBH^{i}(K_v,C\dtp\Z/n)$ lies in
$\bigoplus_{v\notin V}\BBH^{i}(K_v,C\dtp\Z/n)$ for some non-empty open subset $V$ of $U$.
By \cref{sequence: long seq cpt supp and 3 arrows lemma}(3),
we can send $\al$ to $\BBH^{i+1}_c(V,\CC\dtp\Z/n)$ and hence to $\BBH^{i+1}_c(U,\CC\dtp\Z/n)$
by covariant functoriality of $\BBH^{i+1}_c(-,\CC\dtp\Z/n)$.
The following \cmdm for $W\subset V$
\[
\xm{
\bigoplus\limits_{v\notin W}\BBH^i(K_v,C\dtp\Z/n)\ar[r] & \BBH^{i+1}_c(W,\CC\dtp\Z/n)\ar[d]\\
\bigoplus\limits_{v\notin V}\BBH^i(K_v,C\dtp\Z/n)\ar[r]\ar[u] & \BBH^{i+1}_c(W,\CC\dtp\Z/n)
}
\]
shows that the construction does not depend on the choice of $V$.
Finally, the sequence (\ref{sequence: key ex-seq in deg 1 with finite coeff}) is a complex and the square in diagram (\ref{diagram: provides injectivity in degree 1,2 in finite level}) below commutes by the same argument as in the proof of \cite{HSz16}*{Proposition 4.2}.

Conversely, take $\al\in\ker\big(\BBH^{i+1}_c(U,\CC\dtp\Z/n)\to \BBD^{i+1}_K(U,\CC\dtp\Z/n)\big)$.
Let $V\subset U$ be a \neosbk.
We consider the following diagram for $-1\le i\le 1$:
\beac\label{diagram: provides injectivity in degree 1,2 in finite level}
\xm{
\BBH^{i+1}_c(V,\CC\dtp\Z/n)\ar[r] & \BBH^{i+1}_c(U,\CC\dtp\Z/n)\ar[r]\ar[d]
                                  & \bigoplus\limits_{v\in U\setminus V}\BBH^{i+1}(\kappa(v),\CC\dtp\Z/n)\ar[d] \\
                                  & \BBH^{i+1}(K,C\dtp\Z/n)\ar[r]
                                  & \bigoplus\limits_{v\in U\setminus V}\BBH^{i+1}(K_v,C\dtp\Z/n).
}
\eeac
The upper row is exact by \cref{sequence: long seq cpt supp and 3 arrows lemma}(1).
The left vertical arrow is just the composition
\[\BBH^{i+1}_c(U,\CC\dtp\Z/n)\to \BBH^{i+1}(U,\CC\dtp\Z/n)\to \BBH^{i+1}(K,C\dtp\Z/n).\]
The right vertical arrow is given by the composition
\[ \BBH^{i+1}(\kappa(v),\CC\dtp\Z/n)\simeq \BBH^{i+1}(\CO_{v},\CC\dtp\Z/n)\to \BBH^{i+1}(K_v,C\dtp\Z/n).\]
By \cref{lemma: inj of cohomology from Ov to Kv for cpx of tori in finite level}, the right vertical arrow in diagram (\ref{diagram: provides injectivity in degree 1,2 in finite level}) is injective.

Finally, thanks to the exactness of the upper row in diagram (\ref{diagram: provides injectivity in degree 1,2 in finite level}), $\al$ comes from an element $\beta\in\BBH^{i+1}_c(V,\CC\dtp\Z/n)$ by diagram chasing.
Since $\beta$ goes to zero in $\BBH^{i+1}(K,C\dtp\Z/n)$, we may choose $V$ sufficiently small such that $\beta$ already maps to zero in $\BBH^{i+1}(V,\CC\dtp\Z/n)$.
Now the proof is completed by \cref{sequence: long seq cpt supp and 3 arrows lemma}(3).
\end{proof}

Let $\Sha^i(C\dtp\Z/n)\ce\ker\big(\BBH^i(K,C\dtp\Z/n)\to\prod_{v\in X^{(1)}}\BBH^i(K_v,C\dtp\Z/n)\big)$.
Now we construct a \ppair $\Sha^i(C\dtp\Z/n)\times\Sha^{1-i}(C'\dtp\Z/n)\to \QZ$ of finite groups for $i=-1,0$.

\begin{thm}\label{duality: global-finite level}
There is a \ppair of finite groups for $i=-1,0:$
\[\Sha^i(C\dtp\Z/n)\times\Sha^{1-i}(C'\dtp\Z/n)\to \QZ.\]
\end{thm}
\begin{proof}
We define $\rmD^i_{\sh}(U,\CC\dtp\Z/n)$ to be the kernel of the last arrow of the upper row in the following diagram
\[\xm{
0\ar[r] & \rmD^i_{\sh}(U,\CC\dtp\Z/n)\ar[r]\ar@{-->}[d] & \BBH^i(U,\CC\dtp\Z/n)\ar[r]\ar[d] &
\prod\limits_{v\in X^{(1)}}\BBH^i(K_v,C\dtp\Z/n)\ar[d]\\
0\ar[r] & \BBD^{1-i}_K(U,\CC'\dtp\Z/n)^D\ar[r] & \BBH^{1-i}_c(U,\CC'\dtp\Z/n)^D\ar[r] &
\big(\bigoplus\limits_{v\in X^{(1)}}\BBH^{-i}(K_v,C'\dtp\Z/n)\big)^D.
}\]
The middle vertical arrow is an isomorphism by \cref{duality: AV-finite level} and that for the right one by \cref{duality: local-finite level}.
It follows that the left vertical arrow is an isomorphism as well.
Since the group $\BBH^{1-i}_c(U,\CC'\dtp\Z/n)$ is finite
and the functor $\BBH^{1-i}_c(-,\CC'\dtp\Z/n)$ is covariant,
$\{\BBD^{1-i}_K(U,\CC'\dtp\Z/n)\}_{U\subset X_0}$ forms a decreasing family of finite abelian groups.
Hence there exists a non-empty open subset $U_0\subset X_0$
such that $\BBD^{1-i}_K(U,\CC'\dtp\Z/n)=\BBD^{1-i}_K(U_0,\CC'\dtp\Z/n)$ for all non-empty open subset $U\subset U_0$,
i.e. we deduce that $\BBD^{1-i}_K(U,\CC'\dtp\Z/n)\simeq\Sha^{1-i}(C'\dtp\Z/n)$ for all $U\subset U_0$.
Now passing to the direct limit over all $U$ of the isomorphism
$\rmD^i_{\sh}(U,\CC\dtp\Z/n)\simeq \Sha^{1-i}(C'\dtp\Z/n)^D$
yields
$\Sha^i(C\dtp\Z/n)\simeq \drl_n\rmD^i_{\sh}(U,\CC\dtp\Z/n)\simeq \Sha^{1-i}(C'\dtp\Z/n)^D$.
\end{proof}

\begin{rmk}\label{remark: on degrees II}
Actually \cref{duality: global-finite level} yields perfect pairings for $-1\le i\le 2$ by symmetry.
We shall see later \cref{sequence: finite PT seq 15-term} that $\Sha^{-2}(C\dtp\Z/n)=0$.
For $i\le -3$ and $i\ge 3$, we consider the \distri ${_n}T_1[1]\to {_n}T_2[1]\to C\dtp\Z/n\to {_n}T_1[2]$.
Applying \cref{remark: on degrees I}(1), we see that $H^i(K,{_n}T_j)=0$ for $i\ge 4$ and $j=1,2$ by the Kummer sequences.
Now it follows that $\BBH^i(K,C\dtp\Z/n)=0$ for $i\le -3$ and $i\ge 3$ by d\'evissage.
\end{rmk}

\subsection{The Poitou--Tate sequence}
\begin{lem}\label{lemma: localization sequence I}
Let $\CA$ be either $\CC$ or $\CC\dtp\Z/n$ over $U\subset X_0$ and let $A$ be its generic fibre.
If $\al\in \BBH^i(V,\CA)$ is such that $\al_v\in \BBH^i(K_v,A)$ belongs to $\BBH^i(\CO_v,\CA)$ for all $v\in U\setminus V$, then $\al\in \Im\big(\BBH^i(U,\CA)\to \BBH^i(V,\CA)\big)$ for $i\in\Z$.
\end{lem}
\begin{proof}
The localization sequences \cite{Fu11}*{Proposition 5.6.11} for the respective pairs of open immersions $V\subset U$ and $\e K_v\subset \e \CO_v$
(actually here we do the same argument as loc. cit. by replacing injective resolutions by injective Cartan--Eilenberg resolutions)
together with \cite{MilneEC}*{pp.~93, 1.28}
(here we use the $5$-lemma to pass from a single sheaf to a short complex of sheaves)
induce the following \cmdm with exact rows
\[\xm{
\BBH^i(U,\CA)\ar[r]\ar[d] & \BBH^i(V,\CA)\ar[r]\ar[d] &
\bigoplus\limits_{v\in U\setminus V}\BBH^{i+1}_v(\CO_v^h,\CA)\ar[d]\\
\bigoplus\limits_{v\in U\setminus V}\BBH^{i}(\CO_v,\CA)\ar[r] &
\bigoplus\limits_{v\in U\setminus V}\BBH^{i}(K_v,A)\ar[r] &
\bigoplus\limits_{v\in U\setminus V}\BBH^{i+1}_v(\CO_v,\CA).
}\]
By \cite{DH18}*{Lemma 2.6} the right vertical map is an isomorphism, so a diagram chasing yields the desired result.
\end{proof}

\begin{lem}\label{sequence: finite PT sequence row 3 4 5}
There are \ttess for $n\ge 1$ and $-1\le i\le 2$:
\[\BBH^{i}(K,C\dtp\Z/n)\to \BBP^{i}(K,C\dtp\Z/n)\to \BBH^{-i}(K,C'\dtp\Z/n)^D.\]
\end{lem}
\begin{proof}
For $V\subset U\subset X_0$, there exists an \ttes by \cref{lemma: localization sequence I}
\[
\BBH^{i}(U,\CC\dtp\Z/n)
\to
\tst\prod\limits_{v\notin U}\BBH^{i}(K_v,C\dtp\Z/n)
\times
\tst\prod\limits_{v\in U\setminus V}\BBH^{i}(\CO_v,\CC\dtp\Z/n)
\to
\BBH^{i+1}_c(V,\CC\dtp\Z/n).
\]
By \ArtVer duality \ref{duality: AV-finite level},
we obtain an isomorphism $\BBH^{i+1}_c(V,\CC\dtp\Z/n)\simeq \BBH^{-i}(V,\CC'\dtp\Z/n)^D$.
Taking inverse limit over $V$ then yields an \ttes
\[
\BBH^{i}(U,\CC\dtp\Z/n)
\to
\tst\prod\limits_{v\notin U}\BBH^{i}(K_v,C\dtp\Z/n)
\times
\tst\prod\limits_{v\in U}\BBH^{i}(\CO_v,\CC\dtp\Z/n)
\to
\BBH^{-i}(K,C'\dtp\Z/n)^D.
\]
Now we conclude the desired \ttes by taking direct limit over $U$.
\end{proof}

\inpart there is an \ttes
$\BBH^{2}(K,C'\dtp\Z/n)\to \BBP^{2}(K,C'\dtp\Z/n)\to \BBH^{-2}(K,C\dtp\Z/n)^D$
by applying \cref{sequence: finite PT sequence row 3 4 5} to $C'$.
It follows that there are exact sequences
\[\BBH^{-2}(K,C\dtp\Z/n)\to \BBP^{-2}(K,C\dtp\Z/n)\to \BBH^{2}(K,C'\dtp\Z/n)^D\]
by dualizing the above sequence of discrete abelian groups.
Recall that double dual of a finite abelian group is itself.
To close this section, we summarize all the above arithmetic dualities into a $15$-term \ttes as follows.

\begin{thm}\label{sequence: finite PT seq 15-term}
There is a $15$-term \ttes for $n\ge 1$
\beac
\xm@C=20pt @R=12pt{
0\ar[r]
& \BBH^{-2}(K,C\dtp\Z/n)\ar[r] & \BBP^{-2}(K,C\dtp\Z/n)\ar[r] & \BBH^{ 2}(K,C'\dtp\Z/n)^D
\ar@{->} `r/5pt[d] `/8pt[l] `^dl/8pt[lll] `^r/8pt[dl][dll]\\
& \BBH^{-1}(K,C\dtp\Z/n)\ar[r] & \BBP^{-1}(K,C\dtp\Z/n)\ar[r] & \BBH^{ 1}(K,C'\dtp\Z/n)^D
\ar@{->} `r/5pt[d] `/8pt[l] `^dl/8pt[lll] `^r/8pt[dl] [dll]\\
& \BBH^{ 0}(K,C\dtp\Z/n)\ar[r] & \BBP^{ 0}(K,C\dtp\Z/n)\ar[r] & \BBH^{ 0}(K,C'\dtp\Z/n)^D
\ar@{->} `r/5pt[d] `/8pt[l] `^dl/8pt[lll] `^r/8pt[dl] [dll]\\
& \BBH^{ 1}(K,C\dtp\Z/n)\ar[r] & \BBP^{ 1}(K,C\dtp\Z/n)\ar[r] & \BBH^{-1}(K,C'\dtp\Z/n)^D
\ar@{->} `r/5pt[d] `/8pt[l] `^dl/8pt[lll] `^r/8pt[dl] [dll]\\
& \BBH^{ 2}(K,C\dtp\Z/n)\ar[r] & \BBP^{ 2}(K,C\dtp\Z/n)\ar[r] & \BBH^{-2}(K,C'\dtp\Z/n)^D \ar[r] & 0
}
\eeac
\end{thm}
\begin{proof}
The injectivity of the first arrow is a direct consequence of the injectivity of ${_n}T_1(K)\to {_n}T_1(K_v)$.
The surjectivity of the last arrow follows by dualizing the injective map $\BBH^{-2}(K,C'\dtp\Z/n)\to \BBP^{-2}(K,C'\dtp\Z/n)$.

Next, we show that the map $\BBH^i(K,C'\dtp\Z/n)\to \BBP^i(K,C'\dtp\Z/n)$ has discrete image for $-1\le i\le 2$.
Since $\BBP^2(K,C'\dtp\Z/n)$ itself is discrete, there is nothing to do.
For $i=0,\pm1$, suppose $\al\in\Im\big(\BBH^i(K,C'\dtp\Z/n)\to \BBP^i(K,C'\dtp\Z/n)\big)$ lies in
$\prod_{v\notin U}\BBH^i(K_v,C'\dtp\Z/n)\times\prod_{v\in U}\BBH^i(\CO_v,\CC'\dtp\Z/n)$
for some $U\subset X_0$.
Then $\al$ comes from the finite group $\BBH^i(U,\CC'\dtp\Z/n)$ by \cref{lemma: localization sequence I}.
Now dualizing the exact sequences
$0\to \Sha^i(C'\dtp\Z/n)\to \BBH^i(K,C'\dtp\Z/n)\to \BBP^i(K,C'\dtp\Z/n)$
yields the exactness at all the remaining terms.
\end{proof}

\section{Results for the complex \texorpdfstring{$C$}{C}}
\subsection{Global duality: preliminaries}
In this subsection, we establish an \ArtVer style duality and some local duality results.
We begin with a list of properties of abelian groups under consideration.
Recall that $C=[T_1\stra{\rho}T_2]$.

\begin{lem}\label{lemma: list of groups II}
Let $P$ be a $K$-torus that extends to a $U_0$-tori $\CP$ for some sufficiently small non-empty open subset $U_0$ of $X$.
Let $U$ be a non-empty open subset of $U_0$.
Let $L$ be either $K$ or $K_v$.
\benuma
\item\label{list I 1}
The torsion groups $\BBH^1(U,\CC)_{\tors}$ and $\BBH^1_c(U,\CC)_{\tors}$ are of cofinite type.

\item\label{list I 2}
For $i\ge 2$, the groups $\BBH^i(U,\CC)$ and $\BBH^i_c(U,\CC)$ are torsion of cofinite type.

\item\label{list I 3}
The group $H^1(K,P)$ has finite exponent and
the group $H^1(K_v,P)$ is finite.
Moreover, $H^i(L,P)=0$ for $i\ge 3$.
Finally, the groups $\Sha^i(P)$ are finite for each $i\ge 0$.

\item\label{list I 4}
Let $\Phi$ be a group of multiplicative type over $K$.
Then the groups $H^1(L,\Phi)$ and $H^3(L,\Phi)$ have finite exponents.

\item\label{list I 5}
Suppose $M\ce \ker\rho$ is finite.
Then the groups $\BBH\ui(K_v,C)$ have a common finite exponent for all $v\in X\uun$.
Moreover, the groups $\BBH\ui(K,C)$ and $\BBH^1(K,C)$ are torsion of finite exponent.

\item\label{list I 6}
Suppose $T\ce \coker\rho$ is trivial.
The groups $\BBH^0(K,C)$ and $\BBH^2(K,C)$ are torsion of finite exponent.
\eenum
\end{lem}
\pf
For \ref{list I 1}and \ref{list I 2}, see \cite{cat19WA}*{Lemma 1.1}.

\benuma
\item[{\ref{list I 3}}]
The group $H^3(K,P)$ is the direct limit of the groups $H^3(V,\CP)$ for $V\subset U_0$,
but by \cite{SvH03}*{Corollary 4.10} $H^3(V,\CP)=0$ for $V$ sufficiently small and so $H^3(K,P)=0$.
The vanishing of $H^3(K_v,P)$ is explained in \cite{HSz16}*{Remark 2.3} (see also \cite{SerreCG}*{II.5.3}).
For cohomological dimension reasons,
we see that $H^i(L,P)\simeq \drl_n H^i(L,{_n}P)=0$ for $i\ge 4$
where the first isomorphism follows from the Kummer sequence
$0\to H^3(L,P)/n\to H^4(L,{_n}P)\to {_n}H^4(L,P)\to 0$.
The remaining stuffs are proved in \cite{cat19WA}*{Lemma 1.1}.

\item[{\ref{list I 4}}]
Embed $\Phi$ into a \ses $0\to P\to \Phi\to F\to 0$ where $P$ is an $L$-torus and $F$ is a finite \'etale commutative group scheme.
Thus there is an \ttes $H^i(L,P)\to H^i(L,\Phi)\to H^i(L,F)$ for $i\ge 1$.
By d\'evissage, it follows that
$H^1(L,\Phi)$ has finite exponent by Hilbert's Theorem $90$
and
so does $H^3(L,\Phi)$ by $H^3(L,P)=0$.

\item[{\ref{list I 5}}]
Note that there is an isomorphism $\BBH\ui(K_v,C)\simeq \BBH^0(K_v,M)$ thanks to the \distri $M[1]\to C\to T\to M[2]$.
Since $M$ is finite by assumption, $\BBH\ui(K_v,C)$ has a common finite exponent for each $v\in X\uun$.
The group $\BBH\ui(K,C)$ has \finiexp for the same reason.
Thanks to the exact sequence $H^2(K,M)\to \BBH^1(K,C)\to H^1(K,T)$, we deduce that $\BBH^1(K,C)$ have finite exponent by d\'evissage.

\item[{\ref{list I 6}}]
In this case, the short complex $C$ is quasi-isomorphic to $M[1]$.
Thus the desired results follow from {\ref{list I 4}}.
\qed
\eenum

\begin{rmk}\label{remark: the symmetry of finite kernel and surjectivity}
Note that the finiteness of $\ker\rho$ is equivalent to the finiteness of $\cok(\wh{T}_2\to \wh{T}_1)$,
and hence it is equivalent to the injectivity of $\wc{T}_1\to \wc{T}_2$.
Therefore the finiteness of $\ker\rho$ amounts to saying that $\rho':T_2'\to T_1'$ is surjective, and vice versa.
By \cref{lemma: list of groups II}(5,6), we see that
\bitem
\item
If $\ker\rho$ is finite, then $\BBH\ui(K,C)$, $\BBH^0(K,C')$, $\BBH^1(K,C)$ and $\BBH^2(K,C')$ are torsion of finite exponent.

\item
If $\cok\rho$ is trivial, then $\BBH\ui(K,C')$, $\BBH^0(K,C)$, $\BBH^1(K,C')$ and $\BBH^2(K,C)$ are torsion of finite exponent.
\eitem
\end{rmk}

\subsection{Global duality}
The goal of this section is to establish global duality results.
We begin with the finiteness of $\Sha^0(C)$ and $\Sha^2(C)$ (and that of $\Sha^0(C')$ and $\Sha^2(C')$ by symmetry).

\begin{lem}
The groups $\Sha^0(C)$ and $\Sha^2(C)$ are finite.
\end{lem}
\pf
Let $L$ be $K$ or $K_v$ for $v\in X^{(1)}$.
We consider the \distri
\be\label{triangle: kernel-complex-cokernel}
M[1]\to C\to T\to M[2]
\ee over $L$.
By \cref{lemma: list of groups II}(4), the groups $H^1(L,M)$ and $H^3(L,M)$ have finite exponents.
\bitem
\item
The \distri $(\ref{triangle: kernel-complex-cokernel})$ yields \ttess $H^3(L,M)\to \BBH^2(L,C)\to H^2(L,T)\to H^4(L,M)$.
Note 
 that $\Sha^2(T)$ is finite by \cref{lemma: list of groups II}(3).
\inpart $\Sha^2(C)$ has \finiexp by d\'evissage
and it remains to show that $\Sha^2(C)$ is of cofinite type.
Since $\BBH^2(K,C)$ is the direct limit of $\BBH^2(U,\CC)$, each $\al\in\Sha^2(C)$ comes from some $\BBH^2(U,\CC)$ with $U$ a \neos of $X_0$.
\inpart $\al$ lies in the image of $\BBH^2_c(U,\CC)$ by \cref{sequence: long seq cpt supp and 3 arrows lemma}(3).
We conclude that $\al$ comes from $\BBH^2_c(X_0,\CC)$ by the Three Arrows Lemma and hence $\Sha^2(C)$ is a subquotient of $\BBH^2_c(X_0,\CC)$
(which is torsion of cofinite type by \cref{lemma: list of groups II}(2)).
As a consequence, $\Sha^2(C)$ is of cofinite type.
Therefore $\Sha^2(C)$ is finite.

\item
The \ttes $0\to H^1(K,M)\to \BBH^0(K,C)\to H^0(K,T)$ obtained from $(\ref{triangle: kernel-complex-cokernel})$
yields an isomorphism $\Sha^1(M)\simeq\Sha^0(C)$ as $\Sha^0(T)=0$.
An analogous argument as above implies that $\Sha^1(M)\subset\Im (H^1_c(X_0,\CM)\to H^1(K,M))$.
But this map factor through $H^1_c(X_0,\CM)/N\to H^1(K,M)$ because $H^1(K,M)$ has finite exponent for some positive integer $N$.
Finally, $H^1_c(X_0,\CM)/N$ injects into the finite group $H^1_c(X_0,\CM\dtp\Z/N)$ thanks to the Kummer sequence
(more precisely, the Kummer sequence for $[\CT_1\to \CT]\simeq \CM[1]$),
so it is finite as well.
Hence $\Sha^1(M)\simeq \Sha^0(C)$ is contained in this finite image which completes the proof.
\qed
\eitem

\begin{rmk}\label{remark: on degrees III Sha C}
In fact, all non-trivial \TateSha groups of the complex $C$ are finite:
\bitem
\item
$\Sha^1(C)$ is a finite group.
This fact is more complicated because the group $\BBH^1_c(U,\CC)$ needs not to be torsion.
See \cite{cat19WA}*{Proposition 1.12} for a proof.

\item
$\Sha^i(C)=0$ for $i\le -1$ and $i\ge 3$ by similar arguments as \cref{remark: on degrees II}
thanks to the \distri $T_1\to T_2\to C\to T_1[1]$.
\eitem
\end{rmk}

Now we prove a first global duality result between the finite groups $\Sha^0(C)$ and $\Sha^2(C')$.
Let $\BBD^i_K(U,\CC)\ce\Im\big(\BBH^i_c(U,\CC)\to \BBH^i(K,C)\big)$.
The first input is the following \ttes constructed as \cite{cat19WA}*{Proposition 1.15}
\be\label{sequence: key ex-seq in deg 1}
\tstopslim_{v\in X^{(1)}}\BBH^1(K_v,C)\to \BBH^{2}_c(U,\CC)\to \BBD^{2}_K(U,\CC)\to 0.
\ee
In our situation, the exactness is guaranteed by \cref{injectivity: cohohomology of O into K for -1 0 1 2} below instead of \cite{cat19WA}*{Proposition 1.13} there.

\begin{lem}\label{injectivity: cohohomology of O into K for -1 0 1 2}
The homomorphism $\BBH^i(\CO_v,\CC)\to \BBH^i(K_v,\CC)$ induced by the canonical morphism $\e K_v\to \e \CO_v$ is injective for $-1\le i\le 2$.
\end{lem}
\proof\hfill
\benuma
\item $i=-1$.
Since $\CT_1$ is affine (hence separated), $\CT_1(\CO_v)\to T_1(K_v)$ is injective.
It follows that the homomorphism $\BBH\ui(\CO_v,\CC)\to \BBH\ui(K_v,C)$ is injective by d\'evissage thanks to the \distri $T_1\to T_2\to C\to T_1[1]$.
\item $i=0$.
We consider the \distri $M[1]\to C\to T\to M[2]$.
By d\'evissage, it will be sufficient to show $H^1(\CO_v,\CM)\to H^1(K_v,M)$ is injective.
We may realize $\CM$ as an extension $1\to \CP\to \CM\to \CF\to 1$ of a finite group scheme $\CF$ by a torus $\CP$ over $\CO_v$.
Recall that $H^1(\CO_v,\CP)\to H^1(K_v,P)$ and $H^1(\CO_v,\CF)\to H^1(K_v,F)$ are injective,
and that $H^0(\CO_v,\CF)=H^0(K_v,F)$ since $\CF$ is a finite group scheme.
It follows that $H^1(\CO_v,\CM)\to H^1(K_v,M)$ is injective by d\'evissage.
\item $i=1$.
The validity of this case is ensured by \cite{cat19WA}*{Proposition 1.13}.
\item $i=2$.
There is a \cmdm obtained from the respective \Kumseqs
\[
\xm@R15pt{
\drl\limits_n \BBH^1(\CO_v,\CC\dtp\Z/n)\ar[r]\ar[d] & \BBH^2(\CO_v,\CC)\ar[d]\\
\drl\limits_n \BBH^1(K_v,C\dtp\Z/n)\ar[r] & H^2(K_v,C).
}
\]
Since the groups $\BBH^1(K_v,C)$ and $\BBH^1(\kappa(v),\CC)$ are torsion,
we observe horizontal arrows are isomorphisms because of $\BBH^1(K_v,C)\ots\QZ=0$ and
$\BBH^1(\CO_v,\CC)\ots\QZ\simeq\BBH^1(\kappa(v),\CC)\ots\QZ=0$.
Since the left vertical arrow is injective, so is the right one by diagram chasing.
\qed
\eenum

Now we arrive at:

\begin{thm}\label{duality: global 0 2 easy one}
There is a \ppair of finite groups:
\[
\Sha^0(C')\times\Sha^2(C)\to\QZ.
\]
\end{thm}
\begin{proof}
We first identify $\Sha^2(C)$ with the image of $\BBH^2_c(U,\CC)$ in $\BBH^2(K,C)$ for $U$ sufficiently small.
Since $\BBH^2_c(U,\CC)$ is torsion of cofinite type by \cref{lemma: list of groups II}, so is $\BBD^2_K(U,\CC)$.
Hence the decreasing family $\{\BBD^2_K(U,\CC)\{\ell\}\}_{U\subset X_0}$ of $\ell$-primary torsion groups must be stable by \cite{HSz16}*{Lemma 3.7}.
Let us say $\BBD^2_K(U,\CC)\{\ell\}=\BBD^2_K(U_0,\CC)\{\ell\}$ for some open subset $U_0\subset X_0$ and for each non-empty open subset $U\subset U_0$.
Letting $U$ run through all non-empty open subsets of $U_0$, we conclude $\BBD^2_K(U_0,\CC)\{\ell\}=\Sha^2(C)\{\ell\}$ by \cref{sequence: long seq cpt supp and 3 arrows lemma}(3).

Let $\BBD^0_{\sh}(U,\CC')$ be the kernel of $\BBH^0(U,\CC')\to \prod_{v\in X^{(1)}}\BBH^0(K_v,C')$.
We consider the following exact \cmdm
(the lower row is exact because $\ell$-adic completion is right exact by \cite{cat19WA}*{Lemma 1.7})
\[
\xm{
0\ar[r] & \BBD^0_{\sh}(U,\CC')\{\ell\}\ar[r]\ar@{-->}[d]_{\Phi_U} & \BBH^0(U,\CC')\{\ell\}\ar[r]\ar[d] &
\big(\prod\limits_{v\in X^{(1)}}\BBH^0(K_v,C')\big)\{\ell\}\ar[d] \\
0\ar[r] & \big(\BBD^2_K(U,\CC)^{(\ell)}\big)^D\ar[r] & \big(\BBH^2_c(U,\CC)^{(\ell)}\big)^D\ar[r] &
\big(\big(\bigoplus\limits_{v\in X^{(1)}}\BBH^1(K_v,C)\big)^{(\ell)}\big)^D
}
\]
with the left vertical arrow $\Phi_U$ obtained from the commutativity of the right square.
Moreover, the right vertical arrow is an isomorphism by local duality \cite{cat19WA}*{Remark 1.7} and the middle one is surjective with divisible kernel by \cite{cat19WA}*{Proposition 1.2}.
It follows that the left vertical arrow is surjective and $\ker\Phi_U$ is divisible.
\inpart $\drl_U\ker\Phi_U$ is also divisible.
But $\drl_U\BBD^0_{\sh}(U,\CC')\{\ell\}\simeq \Sha^0(C')\{\ell\}$ is finite,
therefore it does not contain non-trivial divisible subgroup, i.e. $\drl_U\ker\Phi_U=0$ is trivial.
Consequently, we obtain the following identifications:
\[
\Sha^0(C')\{\ell\}
\simeq \drl_U\BBD^0_{\sh}(U,\CC')\{\ell\}
\simeq \drl_U \big(\BBD^2_K(U,\CC)^{(\ell)}\big)^D.
\]
Recall that $\Sha^2(C)$ is finite and $\BBD^2_K(U,\CC)\{\ell\}=\Sha^2(C)\{\ell\}$ for $U$ sufficiently small.
So $\BBD^2_K(U,\CC)^{(\ell)}\simeq\Sha^2(C)\{\ell\}^{(\ell)}\simeq \Sha^2(C)\{\ell\}$, and it follows that $\Sha^0(C')\{\ell\}\to \Sha^2(C)\{\ell\}^D$ is an isomorphism.
Since $\Sha^0(C')$ and $\Sha^2(C)$ are finite,
they are finite direct sums of $\ell$-primary parts
and therefore $\Sha^0(C')\times\Sha^2(C)\to\QZ$ is a perfect pairing.
\end{proof}

To connect the first two rows in the Poitou--Tate sequence (\ref{diagram: 12-term PT for short complex}),
we shall need an additional global duality concerning inverse limits.

\begin{lem}\label{injectivity: Pi mod n into product of Hi mod n for -1 and 0}
For $n\ge 1$ and $i\ge -1$,
the canonical \ttai $\BBH^i(\CO_v,\CC)/n\to \BBH^i(K_v,C)/n$ induced by the inclusion $\BBH^i(\CO_v,\CC)\to \BBH^i(K_v,C)$
is injective as well for $v\in X_0\uun$.
\inpart the homomorphism $\BBP^i(K,C)/n\to \prod_{v\in X^{(1)}}\BBH^i(K_v,C)/n$ induced by the inclusion $\BBP^i(K,C)\subset \prod_{v\in X^{(1)}}\BBH^i(K_v,C)$ is injective for $i\ge -1$.
Moreover, the image is the restricted topological product of $\BBH^i(K_v,C)/n$ \wrt the subgroups $\BBH^i(\CO_v,\CC)/n$.
\end{lem}
\begin{proof}
Thanks to the \distri $C\to C\to C\dtp\Z/n \to C[1]$, it suffices to show $\BBH^i(\CO_v,\CC\dtp\Z/n)\to \BBH^i(K_v,C\dtp\Z/n)$ is injective for each $v\in X_0\uun$
which is ensured by \cref{lemma: inj of cohomology from Ov to Kv for cpx of tori in finite level}.
Let $(x_v)\in \BBP^i(K,C)/n$ and let $(\wt{x}_v)\in \BBP^i(K,C)$ be a family of lifts of $x_v\in \BBH^i(K_v,C)/n$ in $\BBH^i(K_v,C)$.
So $\wt{x}_v\in \BBH^i(\CO_v,\CC)$ for all but finitely many $v$.
\inpart its image in $\BBH^i(K_v,C)/n$ lies in the subgroup $\BBH^i(\CO_v,\CC)/n$.
\end{proof}

\begin{thm}\label{duality: global 0 2 hard one}
Put
$\Sha^0_{\wedge}(C)\ce\ker\big(\BBH^0(K,C)_{\wedge}\to \BBP^0(K,C)_{\wedge}\big)$.
If $\ker\rho$ is finite, then there is a \ppair of finite groups
\[
\Sha^0_{\wedge}(C)\times \Sha^2(C')\to \QZ.
\]
\end{thm}

The rest of this section is devoted to the proof of \cref{duality: global 0 2 hard one}
which is analogous to that of \cite{Cyril-these}*{pp.~86-88}.
We proceed by reducing the question into various limits in finite level.

\begin{lem}\label{lemma: an interpretation of Sha wedge}
Let $C=[T_1\stra{\rho} T_2]$ $($here $\ker\rho$ is not necessarily finite$)$.
There is an isomorphism
\[
\Sha^0_{\wedge}(C)
\simeq
\prl_n\Sha^0(C\dtp\Z/n).
\]
\end{lem}
\begin{proof}
Consider the Kummer \ttess
$0\to \BBH^0(K_v,C)/n\to \BBH^0(K_v,C\dtp\Z/n)\to {_n}\BBH^1(K_v,C)\to 0$ for all $v\in X^{(1)}$,
and
$0\to \BBH^0(\CO_v,\CC)/n\to \BBH^0(\CO_v,\CC\dtp\Z/n)\to {_n}\BBH^1(\CO_v,\CC)\to 0$ for all $v\in X_0^{(1)}$.
Moreover, the complex $0\to \BBP^0(K,C)/n\to \BBP^0(K,C\dtp\Z/n)\to {_n}\BBP^1(K,C)\to 0$ is an \ttes
by \cref{injectivity: Pi mod n into product of Hi mod n for -1 and 0}.
Therefore there is a \cmdm with exact rows
by taking inverse limit over all $n$ in the respective \Kumseqs
\beac\label{diagram: Pi mod n into product of Hi mod n for -1 and 0}
\xm{
0\ar[r] & \BBH^0(K,C)_{\wedge}\ar[r]\ar[d] & \prl_n\BBH^0(K,C\dtp\Z/n)\ar[r]\ar[d] & \Phi_K\ar[r]\ar[d] & 0 \\
0\ar[r] & \BBP^0(K,C)_{\wedge}\ar[r] & \prl_n\BBP^0(K,C\dtp\Z/n)\ar[r] & \Phi_{\Pi}\ar[r] & 0.
}
\eeac
where $\Phi_K\subset \prl_n{_n}\BBH^1(K,C)$ and $\Phi_{\Pi}\subset \prl_n{_n}\BBP^1(K,C)$
(here the inverse limit may not be right exact because the involved groups are infinite).
Recall that $\Sha^1(C)$ is finite (\cref{remark: on degrees III Sha C}).
As a consequence, the kernel of the right vertical arrow is contained in $\prl_n{_n}\Sha^1(C)=0$.
Therefore there are isomorphisms
\[
\Sha^0_{\wedge}(C)
\simeq
\ker\big(\prl_n\BBH^0(K,C\dtp\Z/n)\to\prl_n\BBP^0(K,C\dtp\Z/n)\big)
\simeq\prl_n\Sha^0(C\dtp\Z/n)
\]
by the snake lemma, as required.
\end{proof}

The following lemmas tell us that $\prl_n\Sha^0(C\dtp\Z/n)$ is an inverse limit of subgroups of $\BBH^0(U,\CC\dtp\Z/n)$.

\begin{lem}\label{lemma: inverse limit of cohlgy of finite group vanishes}
Let $\CF$ be a finite commutative group scheme over $X_0$.
Then $\prl_n H^i(X_0,{_n}\CF)=0$ for $i\ge 0$.
\end{lem}
\begin{proof}
Since $\CF$ is finite, $\CF={_N}\CF$ for some positive integer $N$.
Take $(x_n)\in \prl_n H^i(X_0,{_{n}}\CF)$.
Then $x_n=Nx_{Nn}$ for each positive integer $n$ and it follows that $x_n=0$,
i.e. $\prl_n H^i(X_0,{_n}\CF)=0$.
\end{proof}

\begin{lem}
Suppose $\ker\rho$ is finite.
Then $\prl_n \BBH^0(X_0,\CC\dtp\Z/{n})\to \prl_n\BBH^0(K,C\dtp\Z/{n})$ is injective.
\end{lem}
\begin{proof}
By \cref{lemma: inverse limit of cohlgy of finite group vanishes},
$\prl_n H^i(X_0,{_n}\CM)=0$.
Consider the \distri ${_{n}}M[2]\to C\dtp\Z/{n}\to T_{\Z/{n}}(C)\to {_{n}}M[3]$.
Thus $\prl_n \BBH^0(X_0,\CC\dtp\Z/{n})\to \prl_n\BBH^0(K,C\dtp\Z/{n})$ is injective by d\'evissage
(recall that $H^1(X_0,T_{\Z/{n}}(\CC))\to H^1(K,T_{\Z/{n}}(C))$ is injective).
\end{proof}

We put
$\BBD^i(U,\CC\dtp\Z/n)\ce\Im\big(\BBH^i_c(U,\CC\dtp\Z/n)\to \BBH^i(U,\CC\dtp\Z/n)\big)$.
So there are inclusions
\[
\prl_n\BBD^i(U,\CC\dtp\Z/n)
\subset \prl_n\BBH^i(U,\CC\dtp\Z/n)
\subset \prl_n\BBH^i(K,C\dtp\Z/n).
\]
If $V\subset U$ is an open subset,
then $\prl_n\BBD^i(V,\CC\dtp\Z/n)\subset \prl_n\BBD^i(U,\CC\dtp\Z/n)$.
In this way, we can take the inverse limit $\prl_U\prl_n\BBD^0(U,\CC\dtp\Z/n)$ over all $U$.

\begin{lem}
Suppose $\ker\rho$ is finite.
There is an isomorphism
\[\prl_U\prl_n\BBD^0(U,\CC\dtp\Z/n)\simeq\prl_n\Sha^0(C\dtp\Z/n)\]
with transition maps given by covariant \hzx of the functor $\BBH^0_c(-,\CC\dtp\Z/n)$.
\end{lem}
\begin{proof}
Since $\prl_n \BBH^0(X_0,\CC\dtp\Z/{n})\to \prl_n\BBH^0(K,C\dtp\Z/{n})$ is injective,
we can take $\bigcap_{U\subset X_0} \prl_n\BBD^0(U,\CC\dtp\Z/n)$ in the inverse limit $\prl_n\BBH^0(K,C\dtp\Z/n)$.
By \cref{sequence: long seq cpt supp and 3 arrows lemma}(3), we conclude $\bigcap_{U\subset X_0} \prl_n\BBD^0(U,\CC\dtp\Z/n)=\prl_n\Sha^0(C\dtp\Z/n)$.
Finally, the intersection over $U\subset X_0$ coincides with inverse limit because of the covariant \hzx of compact support cohomology.
\end{proof}

Next we describe $\Sha^2(C')$.
Again we write $L$ for $K$ or $K_v$ and consider the \Kumseq
$0\to \BBH^1(L,C')/n\to \BBH^1(L,C'\dtp\Z/n)\to {_n}\BBH^2(L,C')\to 0$.
Since $\BBH^i(L,C')$ is torsion for $i\ge 1$, taking the direct limit over all $n$ yields an isomorphism
$\BBH^1(L,\drl_n C'\dtp\Z/n)\simeq\BBH^2(L,C')$.
\inpart we obtain $\Sha^1(\drl_n C'\dtp\Z/n)\simeq\Sha^2(C')$.
Put
\[\BBD^1_{\sh}(U,\drl_n\CC'\dtp\Z/n)\ce
\ker\big(\BBH^1(U,\drl_n\CC'\dtp\Z/n)\to \tstprodlim_{v\in X^{(1)}}\drl\limits_n H^1(K_v,C'\dtp\Z/n)\big).\]
If $V\subset U$ is a smaller open subset, then there is a \ttai
$\BBD^1_{\sh}(U,\drl_n\CC'\dtp\Z/n)\to \BBD^1_{\sh}(V,\drl_n\CC'\dtp\Z/n)$
induced by the restriction maps
$\BBH^1(U,\drl_n\CC'\dtp\Z/n)\to \BBH^1(V,\drl_n\CC'\dtp\Z/n)$.
\inpart the direct limit $\drl_U\BBD^1_{\sh}(U,\drl_n\CC'\dtp\Z/n)$ makes sense and is isomorphic to $\Sha^1(\drl_n C'\dtp\Z/n)$ by construction.
Consequently, we reduce the question to showing that
$\prl_n\BBD^0(U,\CC\dtp\Z/n)\times \BBD^1_{\sh}(U,\drl_n\CC'\dtp\Z/n)\to\QZ$
is a perfect pairing of finite groups.

\vem
We shall need the following compatibility between local duality and Artin--Verdier duality.

\begin{lem}\label{lemma: local duality vs AV duality}
Let $U$ be a sufficiently small non-empty open subset of $X_0$.
There is a \cmdm
\beac\label{diagram: local duality vs AV duality}
\xm{
\tstopslim_{v\in X\uun} \BBH\ui(K_v,C\dtp\Z/n)\ar[d] \ar@{}[r]|-{\bigtimes} &
\tstprodlim_{v\in X\uun}\BBH^1(K_v,C'\dtp\Z/n)\ar[r] & \QZ\ar@{=}[d]\\
\BBH^0_c(U,\CC\dtp\Z/n) \ar@{}[r]|-{\bigtimes} &
\BBH^1(U,\CC'\dtp\Z/n)\ar[u]\ar[r] & \QZ
}
\eeac
where the left vertical arrow is constructed analogous to the first arrow of $(\ref{sequence: key ex-seq in deg 1})$,
and the middle one is the composition
$\BBH^1(U,\CC'\dtp\Z/n)\to \BBH^1(K,C'\dtp\Z/n)\to \prod\BBH^1(K_v,C'\dtp\Z/n)$.
\end{lem}
\begin{proof}
The proof is essentially the same as that of \cite{CTH15}*{Proposition 4.3(f)}.
We first observe by the same argument as \emph{loc. cit.} that it will be sufficient to show the commutativity of diagram (\ref{diagram: local duality vs AV duality}) when $v\notin U$.
Recall there are isomorphisms
$H^1(U,\CC'\dtp\Z/n)\simeq \ext^1_U(\CC\dtp\Z/n,\munots{2})$ and
$H^1(K_v,C'\dtp\Z/n)\simeq \ext^1_{K_v}(C\dtp\Z/n,\munots{2})$.
Hence it suffices to show the following diagram
\[
\xm{
\hom_{K_v}(C\dtp\Z/n,\munots{2}[1]) \ar@{}[r]|-{\bigtimes} &
\BBH\ui(K_v,C\dtp\Z/n)\ar[r]\ar[d] & \QZ\ar@{=}[d]\\
\hom_{U}(\CC\dtp\Z/n,\munots{2}[1]) \ar@{}[r]|-{\bigtimes}\ar[u] &
\BBH^0_c(U,\CC\dtp\Z/n)\ar[r] & \QZ,
}
\]
where the $\hom$ are in the sense of respective derived categories.
Take $\al_U\in\hom_U(\CC\dtp\Z/n,\munots{2}[1])$ and let $\al_v$ be its image in $\hom_{K_v}(C\dtp\Z/n,\munots{2}[1])$.
Let $j_U:U\to X$ and $j_v:\e K_v\to \e \CO_v$ be the respective open immersions.
Recall that there is an isomorphism
$\BBH\ui(K_v,C\dtp\Z/n)\simeq\BBH^0_v(\CO_v,j_{v!}(\CC\dtp\Z/n))$.
In view of the \cmdm
\[
\xm{
\BBH^0_v\big(\CO_v,j_{v!}(\CC\dtp\Z/n)\big)\ar[d]_{(\al_v)_*} &
\BBH^0_v\big(X,j_{U!}(\CC\dtp\Z/n)\big)\ar[l]_-{\simeq}\ar[r]\ar[d]_{(\al_U)_*} &
\BBH^0  \big(X,j_{U!}(\CC\dtp\Z/n)\big)\ar[d]_{(\al_U)_*} \\
\BBH^1_v\big(\CO_v,j_{v!}(\munots{2})\big) &
\BBH^1_v\big(X,j_{U!}(\munots{2})\big)\ar[l]^-{\simeq}\ar[r] &
\BBH^1  \big(X,j_{U!}(\munots{2})\big),
}
\]
we conclude that the desired commutativity of diagram (\ref{diagram: local duality vs AV duality}).
\end{proof}

\begin{lem}
There is a \ppair
\[
\prl_n\BBD^0(U,\CC\dtp\Z/n)\times \BBD^1_{\sh}(U,\drl_n\CC'\dtp\Z/n)\to\QZ.
\]
\end{lem}
\begin{proof}
We consider the following diagram with exact rows
\beac\label{diagram: D-0 and D-1-sh duality}
\xm{
0\ar[r] & \BBD^1_{\sh}(U,\drl_n\CC'\dtp\Z/n)\ar[r] & \BBH^1(U,\drl_n\CC'\dtp\Z/n)\ar[r]\ar[d]_{\simeq} &
\prod\drl_n H^1(K_v,C'\dtp\Z/n) \\
0\ar[r] & \drl_n\big(\BBD^0(U,\CC\dtp\Z/n)\big)^D\ar[r] & \drl_n\big(\BBH_c^0(U,\CC\dtp\Z/n)\big)^D\ar[r] &
\drl_n\big(\bigoplus \BBH\ui(K_v,C\dtp\Z/n)\big)^D\ar[u]
}
\eeac
where the lower row is exact by \cref{proposition: key ex-seq in deg -1 to 1 with finite coeff}
(more precisely, here we identify $\prl_n\BBD^0(U,\CC\dtp\Z/n)$ with its image $\prl_n\BBD^0_K(U,\CC\dtp\Z/n)$ in $\prl_n\BBH^0(K,C\dtp\Z/n)$).
Note that the right vertical arrow in diagram (\ref{diagram: D-0 and D-1-sh duality}) reads as
\[
\drl_n\tstprodlim_v \BBH\ui(K_v,C\dtp\Z/n)^D
\simeq
\drl\limits_n\tstprodlim_v \BBH^1(K_v,C'\dtp\Z/n)
\to
\tstprodlim_v \drl\limits_n\BBH^1(K_v,C'\dtp\Z/n),
\]
and hence it is injective.
The square in diagram (\ref{diagram: D-0 and D-1-sh duality}) commutes by \cref{lemma: local duality vs AV duality}.
Consequently, there is an induced \ttai
\[\drl_n\big(\BBD^0(U,\CC\dtp\Z/n)\big)^D\to \BBD^1_{\sh}(U,\drl_n\CC'\dtp\Z/n)\]
which is actually an isomorphism by diagram chasing.
\end{proof}

\section{Poitou--Tate sequences}
We begin with the topologies on $\BBH^i(K,C)$ and $\BBP^i(K,C)$.
\bitem
\item
For each $i$, the groups $\BBH^i(K,C)$ are endowed with discrete topology.
The groups $\BBH^i(K,C)_{\wedge}$ are endowed with the subspace topology of the product $\prod_{n}\BBH^i(K,C)/n$.
Its topology is not profinite since each component $\BBH^i(K,C)/n$ is not necessary a finite group in general.
\item
For $i=-1,0$,
we give the group $\BBP^i(K,C)_{\wedge}$ induced topology as a subgroup of $\prod \BBH^i(K_v,C)$.
\item
For $i=1,2$,
The group $\BBP^i(K,C)_{\tors}$ is endowed with the direct limit topology.
More precisely, ${_n}\BBP^i(K,C)$ is equipped with the restricted product topology with respect to the discrete topology on each ${_n}\BBH^i(K_v,C)$, and their direct limit $\BBP^i(K,C)_{\tors}$ is equipped with the corresponding direct limit topology.
\eitem

Now we arrive the main result of this paper.

\begin{thm}\label{theorem: PT for complex of tori}
Suppose either $\ker\rho$ is finite or $\cok\rho$ is trivial.
Then there is a $12$-term \ttes of topological abelian groups
\[
\xm@C=20pt @R=12pt{
0\ar[r] & \BBH\ui(K,C)_{\wedge}\ar[r] & \BBP\ui(K,C)_{\wedge}\ar[r] & \BBH^2(K,C')^D
\ar@{->} `r/5pt[d] `/8pt[l] `^dl/8pt[lll] `^r/8pt[dl][dll]\\
& \BBH^0(K,C)_{\wedge}\ar[r] & \BBP^0(K,C)_{\wedge}\ar[r] & \BBH^1(K,C')^D
\ar@{->} `r/5pt[d] `/8pt[l] `^dl/8pt[lll] `^r/8pt[dl] [dll]\\
& \BBH^1(K,C)\ar[r] & \BBP^1(K,C)_{\tors}\ar[r] & \big(\BBH^0(K,C')_{\wedge}\big)^D
\ar@{->} `r/5pt[d] `/8pt[l] `^dl/8pt[lll] `^r/9pt[dl] [dll]\\
& \BBH^2(K,C)\ar[r] & \BBP^2(K,C)_{\tors}\ar[r] & \big(\BBH\ui(K,C')_{\wedge}\big)^D\ar[r] & 0
}
\]
%
\end{thm}

The proof of the theorem consists of several steps.
We first establish perfect pairings between the restricted topological products for any short complex $C$.
Subsequently, we deduce the exactness of the first and the last rows again for any $C$.
Finally, we deal with the more complicated exact sequence in the middle of the diagram
with either $\ker\rho$ being finite or $\cok\rho$ being trivial.
\vem

\noindent\textbf{Step 1: dualities between restricted topological products}.

\vem
We proceed as in the finite level to obtain pairings between $\BBP^i(K,C)_{\wedge}$ and $\BBP^{1-i}(K,C')_{\tors}$ for $i=-1,0$.
Recall \cref{injectivity: Pi mod n into product of Hi mod n for -1 and 0} that
$\BBH^i(\CO_v,\CC)/n\to \BBH^i(K_v,C)/n$ is injective for each $v\in X_0\uun$.
Therefore we are allowed to identify $\BBH^i(\CO_v,\CC)_{\wedge}$ with a subgroup of $\BBH^i(K_v,C)_{\wedge}$ for $v\in X_0\uun$ by the left exactness of inverse limits.
In this step, all the conclusions are valid without any assumption on $\ker\rho$ and $\cok\rho$.

\begin{prop}\label{duality: annihilators of local pairing}
For $i=-1,0$, the annihilator of $\BBH^{1-i}(\CO_v,\CC')$ is $\BBH^i(\CO_v,\CC)_{\wedge}$ under the \ppair
\[
\BBH^i(K_v,C)_{\wedge}\times \BBH^{1-i}(K_v,C')\to \QZ.
\]
\end{prop}
\begin{proof}
We consider the following \cmdm with exact rows for $i=-1,0$
\[\xm{
0\ar[r] & \BBH^i(\CO_v,\CC)/n\ar[r]\ar[d] & \BBH^i(\CO_v,\CC\dtp\Z/n)\ar[r]\ar[d] & \BBH^{i+1}(\CO_v,\CC)\ar[d]\\
0\ar[r] & \BBH^i(K_v,C)/n\ar[r] & \BBH^i(K_v,C\dtp\Z/n)\ar[r] & \BBH^{i+1}(K_v,C).
}\]
Take $t\in \BBH^i(K_v,C)/n$ such that $t$ is orthogonal to ${_n}\BBH^{1-i}(\CO_v,\CC')$.
Then the image $s$ of $t$ in $\BBH^i(K_v,C\dtp\Z/n)$ is orthogonal to $\BBH^{-i}(\CO_v,\CC'\dtp\Z/n)$ and it follows that $s\in \BBH^i(\CO_v,\CC\dtp\Z/n)$ by \cref{duality: annihilators of local pairing in finite level}.
The right vertical arrow is injective by \cref{injectivity: cohohomology of O into K for -1 0 1 2},
thus a diagram chasing shows that $t$ lies in $\BBH^i(\CO_v,\CC)/n$.
\end{proof}

\begin{cor}
For $i=-1,0$, there are isomorphisms
$\BBP^i(K,C)_{\wedge}\simeq \big(\BBP^{1-i}(K,C')_{\tors}\big)^D$ of locally compact groups.
\end{cor}
\begin{proof}
This is an immediate consequence of \cref{injectivity: Pi mod n into product of Hi mod n for -1 and 0} and
\cref{duality: annihilators of local pairing}.
\end{proof}

\vspace{0.5em}
\noindent\textbf{Step 2: exactness of the first and the last rows of diagram (\ref{diagram: 12-term PT for short complex})}.
\vspace{0.5em}

In this step, all the conclusions are valid without any assumption on $\ker\rho$ and $\cok\rho$.

\begin{prop}\label{result: PT seq row 4}
There is an \ttes of locally compact groups
\be\label{sequence: PT seq last row}
0\to \Sha^2(C)\to \BBH^2(K,C)\to \BBP^2(K,C)_{\tors}\to \big(\BBH\ui(K,C')_{\wedge}\big)^D\to 0.
\ee
\end{prop}
\begin{proof}
We consider the following \cmdm with exact rows and exact middle column by \cref{sequence: finite PT seq 15-term}:
\beac\label{diagram: PT seq last row}
\xm{
0\ar[r] & \BBH^1(K,C)/n\ar[r]\ar[d] & \BBH^1(K,C\dtp\Z/n)\ar[r]\ar[d] & {_n}\BBH^2(K,C)\ar[r]\ar[d] & 0 \\
0\ar[r] & \BBP^1(K,C)/n\ar[r]\ar[d] & \BBP^1(K,C\dtp\Z/n)\ar[r]\ar[d] & {_n}\BBP^2(K,C)\ar[r]\ar[d] & 0 \\
0\ar[r] & \big({_n}\BBH^0(K,C')\big)^D\ar[r] & \BBH\ui(K,C'\dtp\Z/n)^D\ar[r] & \big(\BBH\ui(K,C')/n\big)^D\ar[r] & 0.
}
\eeac
Taking direct limit over all $n$ of the last two columns in diagram (\ref{diagram: PT seq last row}) yields the \cmdm
\beac\label{diagram: PT seq last row proof I}
\xm{
\drl_n\BBH^1(K,C\dtp\Z/n)\ar[r]\ar[d] & \drl_n\BBP^1(K,C\dtp\Z/n)\ar[r]\ar[d] &
\big(\prl_n\BBH\ui(K,C'\dtp\Z/n)\big)^D\ar[r]\ar[d] & 0 \\
\BBH^2(K,C)\ar[r] & \BBP^2(K,C)_{\tors}\ar[r] & \big(\BBH\ui(K,C')_{\wedge}\big)^D.
}
\eeac
Since $\BBH^2(K,C)$ is torsion and $\BBH^3(K,C)=0$,
taking direct limit of the \Kumseq
\[
0\to \BBH^2(K,C)/n\to \BBH^2(K,C\dtp\Z/n)\to {_n}\BBH^3(K,C)\to 0.
\]
yields $\drl_n \BBH^2(K,C\dtp\Z/n)=0$.
Taking direct limit in \cref{sequence: finite PT seq 15-term} yields the exactness of the upper row.
The left vertical arrow in (\ref{diagram: PT seq last row proof I}) is an isomorphism since $\BBH^1(K,C)\ots\QZ=0$
and the middle one is surjective by the exactness of $\drl_n$.
If the right vertical arrow is an isomorphism, then a diagram chasing yields the \ttes (\ref{sequence: PT seq last row}).

So it remains to show $\prl_n\BBH\ui(K,C'\dtp\Z/n)\simeq \BBH\ui(K,C')_{\wedge}$.
We see that the vanishing $\prl_n {_n}\BBH^0(K,C')=0$ is enough because of the \Kumseq
\[
0\to \BBH\ui(K,C')/n\to \BBH\ui(K,C'\dtp\Z/n)\to {_n}\BBH^0(K,C')\to 0,
\]
and therefore we reduce to show $\BBH^0(K,C')_{\tors}$ has finite exponent.
Indeed, since $H^0(K,\gm)_{\tors}=(K\uu)_{\tors}=(k\uu)_{\tors}$ is finite,
we see that $H^0(K,P)_{\tors}$ has finite exponent for a $K$-torus $P$
by a restriction-corestriction argument.
The \distri $\ker\rho'[1]\to C'\to \cok\rho'\to \ker\rho'[2]$ yields an \ttes
$0\to H^1(K,\ker\rho')_{\tors}\to \BBH^0(K,C')_{\tors}\to \BBH^0(K,\cok\rho')_{\tors}$.
By \cref{lemma: list of groups II}(4) $H^1(K,\ker\rho')$ has finite exponent,
so is $\BBH^0(K,C')_{\tors}$ by d\'evissage.
\end{proof}

\begin{cor}\label{result: PT seq row 1}
There is an exact sequence of locally compact groups
\[0\to \BBH\ui(K,C)_{\wedge}\to \BBP\ui(K,C)_{\wedge}\to \BBH^2(K,C')^D\to \Sha^2(C')^D\to 0.\]
\end{cor}
\begin{proof}
Applying \cref{result: PT seq row 4} to $C'$ yields an \ttes
\be\label{sequence: PT seq last row for C'}
\BBH^2(K,C')\to \BBP^2(K,C')_{\tors}\to \big(\BBH\ui(K,C)_{\wedge}\big)^D\to 0.
\ee
Note that for each $n\ge 1$, the double dual of the finite discrete group $\BBH\ui(K,C)/n$ is itself.
It follows that the desired sequence is exact at the first three terms by dualizing the sequence (\ref{sequence: PT seq last row for C'}).
We get the exactness at the last three terms by an argument analogous to that of \cite{HSSz15}*{Theorem 2.9} except that we use \cref{lemma: localization sequence I} instead of \cite{HSSz15}*{Lemma 2.8}.
Let us recall the argument for the convenience of the reader.
Let $\BI\ce \Im\big(\BBH^2(K,C')\to \BBP^2(K,C')_{\tors}\big)$.
For each $n\ge 1$, let $A_n\subset \BBH^2(K,C')$ be the inverse image of
$B_n=\prod_{v\notin U}{_n}\BBH^2(K_v,C')\times\prod_{v\in U}{_n}\BBH^2(\CO_v,\CC')\subset \BBP^2(K,C')_{\tors}$.
By \cref{lemma: localization sequence I}, $A_n\subset \Im\big(\BBH^2(U,\CC')\to \BBH^2(K,C')\big)$.
But $\BBH^2(U,\CC')$ is a torsion group of cofinite type by \cref{lemma: list of groups II},
so is $A_n$ and hence $A_n/n$ is finite.
\inpart the image of $A_n$ in the $n$-torsion group $B_n$ is finite.
Since this image is just $\BI\cap B_n$, we obtain that $\BI$ is discrete in $\BBP^2(K,C')_{\tors}$ by definition of the topology on $\BBP^2(K,C')_{\tors}$.
\end{proof}

\vspace{0.5em}
\noindent\textbf{Step 3.1:
Exactness of the second and the third rows of diagram (\ref{diagram: 12-term PT for short complex}): finite kernel case}
\vem

We will systematically assume that $M\ce\ker\rho$ is finite from \cref{lemma: exactness of proj lim in finite level for row 2 and 3} to \cref{result: PT seq row 3}.

\begin{prop}\label{lemma: exactness of proj lim in finite level for row 2 and 3}
There is an \ttes of locally compact groups
\[\prl_n\BBH^0(K,C\dtp\Z/n)\to \prl_n\BBP^0(K,C\dtp\Z/n)\to\BBH^0(K,\drl_nC'\dtp\Z/n)^D.\]
\end{prop}
\begin{proof}
The proof is similar to \cite{Dem11}*{Lemme 6.2}.
We consider the following \cmdm
\beac\label{diagram: 3 times 5 for PT seq second row}
\xm@C=12pt{
H^0(K,T_{\Z/n}(C))\ar[r]^-{\Phi^{K}_n}\ar[d] & H^2(K,{_n}M)\ar[r]\ar[d] & \BBH^0(K,C\dtp\Z/n)\ar[r]\ar[d] &
H^1(K,T_{\Z/n}(C))\ar[r]^-{\Psi^{K}_n}\ar[d] & H^3(K,{_n}M)\ar[d] \\
\BBP^0(K,T_{\Z/n}(C))\ar[r]^-{\Phi^{\Pi}_n}\ar[d] & \BBP^2(K,{_n}M)\ar[r]\ar[d] & \BBP^0(K,C\dtp\Z/n)\ar[r]\ar[d] &
\BBP^1(K,T_{\Z/n}(C))\ar[r]^-{\Psi^{\Pi}_n}\ar[d] & \BBP^3(K,{_n}M)\ar[d] \\
H^3(K,T_{\Z/n}(C)')^D\ar[r] & H^1(K,({_n}M)')^D\ar[r] & \BBH^0(K,C'\dtp\Z/n)^D\ar[r] &
H^2(K,T_{\Z/n}(C)')^D\ar[r] & H^0(K,({_n}M)')^D
}
\eeac
where
the upper two rows are exact since they are induced by the \distri ${_n}M[2]\to C\dtp\Z/n\to T_{\Z/n}(C)[1]\to {_n}M[3]$,
and the columns except the middle one are exact by \cite{HSSz15}*{Theorem 2.3}.
The product $\BBP^0(K,T_{\Z/n}(C))=\prod_{v\in X^{(1)}}\BBH^0(K_v,T_{\Z/n}(C))$ is compact
since $H^0(K_v,T_{\Z/n}(C))$ is finite.
Now taking inverse limit of the middle three columns of diagram (\ref{diagram: 3 times 5 for PT seq second row}) over all $n$ yield the following \cmdm
\beac\label{diagram: 3 times 4 for PT seq second row}
\xm@C=14pt{
\prl_n H^2(K,{_n}M)\ar[r]\ar[d] & \prl_n H^0(K,C\dtp\Z/n)\ar[r]\ar[d] &
\prl_n \ker\Psi_n^{K}\ar[r]\ar[d] & \prl_n^1 \cok\Phi_n^{K}\ar[d] \\
\prl_n \BBP^2(K,{_n}M)\ar[r]\ar[d] & \prl_n \BBP^0(K,C\dtp\Z/n)\ar[r]\ar[d] &
\prl_n \ker\Psi_n^{\Pi}\ar[r]\ar[d] & \prl_n^1 \cok\Phi_n^{\Pi}\ar[d] \\
\prl_n H^1(K,({_n}M)')^D\ar[r] & \prl_n H^0(K,C'\dtp\Z/n)^D\ar[r] &
\prl_n H^2(K,T_{\Z/n}(C)')^D\ar[r] & \prl_n^1 H^0(K,{_n}M')^D.
}
\eeac
Moreover, the third column of the above diagram (\ref{diagram: 3 times 4 for PT seq second row}) also fits into the \cmdm
\beac\label{diagram: 2 times 3 for PT seq second row}
\xm{
0\ar[r] & \prl_n \ker\Psi_n^{K}\ar[r]\ar[d] & \prl_n H^1(K,T_{\Z/n}(C))\ar[r]\ar[d] & \prl_n \Im\Psi_n^{K}\ar[d]\\
0\ar[r] & \prl_n \ker\Psi_n^{\Pi}\ar[r] & \prl_n \BBP^1(K,T_{\Z/n}(C))\ar[r]\ar[d] & \prl_n \Im\Psi_n^{\Pi}\\
& & \prl_n H^2(K,T_{\Z/n}(C)')^D
}
\eeac
with $\prl_n \Im\Psi_n^{K}\subset\prl_n H^3(K,{_n}M)=0$ being zero.
Therefore $\prl_n \ker\Psi_n^{K}\to \prl_n H^1(K,T_{\Z/n}(C))$ is surjective.
Note that the middle column is exact because $\prl_n^1$ of the finite groups
$\ker\big(H^1(K,T_{\Z/n}(C))\to \BBP^1(K,T_{\Z/n}(C))\big)$ vanishes according to \cite{Jen72}*{Th\'eor\`eme 7.3}.

Take $\al\in\prl_n\BBP^0(K,C\dtp\Z/n)$ such that it goes to zero in $\prl_n\BBH^0(K,C'\dtp\Z/n)^D$.
By \hzxbk, $\beta\ce\Im\al\in\prl_n\BBP^1(K,T_{\Z/n}(C))$ goes to zero in $\prl_n H^2(K,T_{\Z/n}(C)')^D$.
Thus $\beta$ comes from $\gamma\in\prl_n H^1(K,T_{\Z/n}(C))$ by the exactness of the middle column in diagram (\ref{diagram: 2 times 3 for PT seq second row}).
Since $\prl_n\ker\Psi_n^{K}\to \prl_nH^1(K,T_{\Z/n}(C))$ is surjective,
$\gamma$ comes from some $\gamma'\in\prl_n\ker\Psi_n^{K}$.
But $\prl^1_{n}\cok\Phi_n^{K}\to \prl^1_{n}\cok\Phi_n^{\Pi}$ is injective by \cref{lemma: inj of derived proj lim} below,
so $\gamma'$ comes from $\tau\in\prl_n H^1(K,T_{\Z/n}(C))$.
Since $\al$ and $\Im\tau$ have the same image in $\prl_n\ker\Psi_n^{\Pi}$ by construction,
the commutativity of diagram (\ref{diagram: 3 times 4 for PT seq second row}) implies that $\al$ and $\Im\tau$ differs from an element in $\prl_n\BBP^2(K,{_n}M)$.
Recall that $\ker\rho$ is finite, thus $\prl_n\BBP^2(K,{_n}M)=0$ and $\al$ comes from $\prl_n H^1(K,T_{\Z/n}(C))$.
\end{proof}

\begin{lem}\label{lemma: inj of derived proj lim}
The homomorphism $\prl^1_{n}\cok\Phi_n^{K}\to \prl^1_{n}\cok\Phi_n^{\Pi}$ in diagram $(\ref{diagram: 3 times 4 for PT seq second row})$ is an isomorphism.
\end{lem}
\begin{proof}
Since $H^0(L,T_{\Z/n}(C))$ is finite for $L=K$, $K_v$,
$\BBP^0(K,T_{\Z/n}(C))=\prod_{v\in X\uun}H^0(K_v,T_{\Z/n}(C))$ is compact.
Thus we obtain $\prl_n^1 H^0(K,T_{\Z/n}(C))=0$ and $\prl_n^1 \BBP^0(K,T_{\Z/n}(C))=0$ by \cite{Jen72}*{Th\'eor\`eme 7.3}.
Moreover, the image $\Im\Phi_n^K$ of $H^0(K,T_{\Z/n}(C))$ in $H^2(K,{_n}M)$ is finite,
so $\prl_n^1H^2(K,{_n}M)\simeq\prl_n^1\cok\Phi_n^K$
by the \ses $0\to \Im\Phi_n^K\to H^2(K,{_n}M)\to \cok\Phi_n^K\to 0$.
Similarly, we obtain $\prl_n^1\BBP^2(K,{_n}M)\simeq\prl_n^1\cok\Phi_n^{\Pi}$.

Let $\rmI_n$ denote the image of $H^2(K,{_n}M)\to \BBP^2(K,{_n}M)$.
So $0\to \Sha^2({_n}M)\to H^2(K,{_n}M)\to \rmI_n\to 0$ is an exact sequence.
The finiteness of $\Sha^2({_n}M)$ yields an isomorphism
$\tst\prl_n^1 H^2(K,{_n}M)\simeq\prl_n^1 \rmI_n$.
Moreover, the cokernel of $\rmI_n\to \BBP^2(K,{_n}M)$ is a subgroup of the group $H^1(K,({_n}M)')^D$
by diagram (\ref{diagram: 3 times 5 for PT seq second row}).
The finiteness of $M$ yields the vanishing of $\prl_n \big(H^1(K,({_n}M)')^D\big)\simeq H^1(K,\drl_n({_n}M)')^D$
and thus $\prl_n\cok\big(\rmI_n\to \BBP^2(K,{_n}M)\big)=0$ by the left exactness of inverse limits.
We conclude that $\prl_n^1 \rmI_n\to\prl_n^1 \BBP^2(K,{_n}M)$ is an isomorphism.
Therefore we have
$\prl^1_n\cok\Phi_n^K\simeq\prl^1_n\rmI_n\simeq\prl^1_n\cok\Phi_n^{\Pi}$.
\end{proof}

\begin{cor}\label{result: PT seq row 2}
There is an \ttes of locally compact groups
\[
0\to \Sha^0_{\wedge}(C)\to \BBH^0(K,C)_{\wedge}\to \BBP^0(K,C)_{\wedge}\to \BBH^1(K,C')^D\to \Sha^1(C')^D\to 0.
\]
\end{cor}
\begin{proof}
We consider again the
diagram (\ref{diagram: Pi mod n into product of Hi mod n for -1 and 0})
with rows and the middle column being exact
(by \cref{lemma: exactness of proj lim in finite level for row 2 and 3}):
\[
\xm{
0\ar[r] & \BBH^0(K,C)_{\wedge}\ar[r]\ar[d] & \prl_n \BBH^0(K,C\dtp\Z/n)\ar[r]\ar[d] & \Phi_K\ar[r]\ar[d] & 0 \\
0\ar[r] & \BBP^0(K,C)_{\wedge}\ar[r]\ar[d] & \prl_n \BBP^0(K,C\dtp\Z/n)\ar[r]\ar[d] & \Phi_{\Pi}\ar[r] & 0. \\
        & \BBH^1(K,C')^D\ar[r]             & \BBH^0(K,\drl_n C'\dtp\Z/n)^D
}
\]
Since $\BBH^0(K,C'\dtp\Z/n)\to {_n}\BBH^1(K,C')$ is a surjective map between finite discrete groups,
there is an injective map $\big({_n}\BBH^1(K,C')\big)^D\to \big(\BBH^0(K,C'\dtp\Z/n)\big)^D$
and hence an injection $\BBH^1(K,C')\to \BBH^0(K,\drl_n C'\dtp\Z/n)$.
So the left vertical column is a complex.
But the map $\Phi_K\to \Phi_{\Pi}$ is injective
(see the proof of \cref{lemma: an interpretation of Sha wedge}),
another diagram chasing then tells us the left column is exact.

To show the exactness of the last three terms, we consider the \ttes $0\to \Sha^1(C')\to \BBH^1(K,C')\to \BBP^1(K,C')_{\tors}$.
It suffices to show $\BBH^1(K,C')\to \BBP^1(K,C')_{\tors}$ has discrete image
(the same reason as the proof of \cref{result: PT seq row 1}).
Since $\ker\rho$ is finite, $\wh{T}_2\to \wh{T}_1$ has finite cokernel and hence $\wc{T}_1\to \wc{T}_2$ is injective.
By \cite{Diego-these}*{Lemme 1.4.10(iv)}, $\BBH^1(U,\CC')$ is torsion of cofinite type for $U\subset X_0$ sufficiently small.
Thus the same argument as \cref{result: PT seq row 1} yields the desired exactness.
\end{proof}

\begin{rmk}
By \cref{lemma: list of groups II}, the groups $\BBH\ui(K,C)$ and $\BBP\ui(K,C)$ are torsion of finite exponent,
therefore they are isomorphic to respective completions $\BBH\ui(K,C)_{\wedge}$ and $\BBP\ui(K,C)_{\wedge}$.
Again by \cref{lemma: list of groups II}, $\BBH^1(K,C)$ and $\BBP^1(K,C)$ have finite exponents.
Consequently, we obtain $\BBH^1(K,C)\simeq \BBH^1(K,C)_{\wedge}$ and $\BBP^1(K,C)_{\tors}=\BBP^1(K,C)\simeq \BBP^1(K,C)_{\wedge}$.
Moreover, the finiteness of $M$ implies that $\rho':T_2'\to T_1'$ is surjective and
hence $\BBH^0(K,C')$ has finite exponent by \cref{lemma: list of groups II}.
Summing up, the subscripts "$\tors$" and "$\wedge$" in the first and the third rows of (\ref{diagram: 12-term PT for short complex}) are superfluous.
\end{rmk}

\begin{prop}\label{result: PT seq row 3}
There is an \ttes of locally compact groups
\[
0\to \Sha^1(C)\to \BBH^1(K,C)\to \BBP^1(K,C)_{\tors}\to \big(\BBH^0(K,C')_{\wedge}\big)^D\to \BBH^2(K,C).
\]
\end{prop}
\begin{proof}
Taking inverse limit in diagram (\ref{diagram: PT seq last row}) yields a \cmdm
\[
\xm{
0\ar[r] & \BBH^1(K,C)\ar[r]\ar[d] & \prl_n \BBH^1(K,C\dtp\Z/n)\ar[r]\ar[d] & \Psi_{K}  \ar[r]\ar[d] & 0 \\
0\ar[r] & \BBP^1(K,C)\ar[r]\ar[d] & \prl_n \BBP^1(K,C\dtp\Z/n)\ar[r]\ar[d] & \Psi_{\Pi}\ar[r] & 0. \\
0\ar[r] & \BBH^0(K,C')^D\ar[r] & \prl_n \BBH\ui(K,C'\dtp\Z/n)^D
}
\]
We observe that $\Psi_{K}\to \Psi_{\Pi}$ is injective because its kernel is contained in $\prl_n{_n}\Sha^2(C)=0$ (recall that $\Sha^2(C)$ is finite).
The exactness of the middle column follows from the vanishing $\prl^1_n\Sha^1(C\dtp\Z/n)=0$
(by \cite{Jen72}*{Th\'eor\`eme 7.3}).
Thus a diagram chasing yields the exactness of the left column.

Now we verify the exactness of $\BBP^1(K,C)\to \BBH^0(K,C')^D\to \BBH^2(K,C)$.
Consider the \cmdm with vertical arrows obtained from respective Kummer sequences:
\beac\label{diagram: PT seq row 3}
\xm{
\BBP^0(K,C\dtp\Z/n)\ar[r]\ar@{->>}[d] & \BBH^0(K,C'\dtp\Z/n)^D\ar[r]\ar@{->>}[d] & \BBH^1(K,C\dtp\Z/n)\ar@{->>}[d]\\
{_n}\BBP^1(K,C)\ar[r] & \big(\BBH^0(K,C')/n\big)^D\ar[r]^-{*} & {_n}\BBH^2(K,C),
}
\eeac
where
the upper row is exact by \cref{sequence: finite PT seq 15-term} and
the arrow $*$ is the composite
$\big(\BBH^0(K,C')/n\big)^D\to \big(\Sha^0(C')/n\big)^D\simeq {_n}\Sha^2(C)\to {_n}\BBH^2(K,C)$.
The left square in (\ref{diagram: PT seq row 3}) commutes by \cref{duality: local-finite level} and
the right one commutes by construction and \cref{duality: global 0 2 easy one}.
Finally, the middle vertical arrow is surjective
because $\BBH^0(K,C'\dtp\Z/n)$ and $\BBH^0(K,C')/n$ are discrete.
Passing to the direct limit over all $n$, the right vertical arrow in diagram (\ref{diagram: PT seq row 3}) becomes an isomorphism as $\BBH^1(K,C)\ots\Q/\Z=0$.
Now a diagram chasing implies the exactness of $\BBP^1(K,C)\to \BBH^0(K,C')^D\to \BBH^2(K,C)$.
\end{proof}

\vspace{0.5em}
\noindent\textbf{Step 3.2:
Exactness of the second and the third rows of diagram (\ref{diagram: 12-term PT for short complex}): the surjective case}
\vem

Suppose $\rho:T_1\to T_2$ is surjective.
Thus $\wh{T}_2\to \wh{T}_1$ is injective and $\wc{T}_1\to \wc{T}_2$ has finite cokernel, i.e. $\rho':T_2'\to T_1'$ has finite kernel.
Therefore $\BBH\ui(K,C')$, $\BBH^0(K,C)$, $\BBH^1(K,C')$ and $\BBH^2(K,C)$
are torsions group having finite exponent by \cref{lemma: list of groups II}.
It follows that $\BBH^0(K,C)_{\wedge}=\BBH^0(K,C)$,
$\BBP^0(K,C)_{\wedge}=\BBP^0(K,C)$,
$\BBP^2(K,C)_{\tors}=\BBP^2(K,C)$ and
$\BBH\ui(K,C')_{\wedge}=\BBH\ui(K,C')$,
i.e. the subscripts in diagram (\ref{diagram: 12-term PT for short complex}) are superfluous.
After \cref{result: PT seq row 1} and \cref{result: PT seq row 4},
it remains to show the following proposition.

\begin{prop}\label{result: PT seq row 2 and 3 surjective case}
Suppose $\rho:T_1\to T_2$ is surjective.
Then there are exact sequences
$0\to \Sha^0(C)\to \BBH^0(K,C)\to \BBP^0(K,C)\to \BBH^1(K,C')^D\to \BBH^1(K,C)\to \BBP^1(K,C)_{\tors}\to \big(\BBH^0(K,C')_{\wedge}\big)^D\to \BBH^2(K,C)$.
\end{prop}
\pff
\bitem
\itm
The exactness of $\BBH^0(K,C)\to \BBP^0(K,C)\to \BBH^1(K,C')^D$ follows from the same argument as the first paragraph of the proof of \cref{result: PT seq row 2}.

\itm
We show the exactness of $\BBH^1(K,C)\to \BBP^1(K,C)_{\tors}\to \big(\BBH^0(K,C')_{\wedge}\big)^D$.
Since $\BBH^1(K,C')$ has finite exponent,
$\prl_n {_n}\BBH^1(K,C')=0$ and hence $\BBH^0(K,C')_{\wedge}\simeq\prl_n \BBH^0(K,C'\dtp\Z/n)$ by the Kummer sequence
$0\to \BBH^0(K,C')/n\to \BBH^0(K,C'\dtp\Z/n)\to {_n}\BBH^1(K,C')\to 0$.
Consider the \cmdm
\[
\xm{
\drl_n \BBH^0(K,C\dtp\Z/n)\ar[r]\ar@{->>}[d] & \drl_n \BBP^0(K,C\dtp\Z/n)\ar[r]\ar@{->>}[d] & \big(\prl_n\BBH^0(K,C'\dtp\Z/n)\big)^D\ar[d]^{\simeq}\\
\BBH^1(K,C)\ar[r] & \BBP^1(K,C)_{\tors}\ar[r] & \big(\BBH^0(K,C')_{\wedge}\big)^D.
}
\]
Now the exactness of the lower row
follows from that of
the upper row by diagram chasing.

\itm
We consider the following \cmdm (see (\ref{diagram: PT seq row 3})) for $i=-1,0$:
\beac\label{diagram: PT seq row 3 surjective case}
\xm{
\BBP^i(K,C\dtp\Z/n)\ar[r]\ar@{->>}[d] & \BBH^{-i}(K,C'\dtp\Z/n)^D\ar[r]\ar@{->>}[d] & \BBH^{i+1}(K,C\dtp\Z/n)\ar@{->>}[d]\\
{_n}\BBP^{i+1}(K,C)\ar[r] & \big(\BBH^{-i}(K,C')/n\big)^D\ar[r] & {_n}\BBH^{i+2}(K,C).
}
\eeac
Note that $\BBH^0(K,C)\simeq H^1(K,M)$ and $\BBH^1(K,C)$ are torsion.
Thus $\BBH^i(K,C)\ots\QZ=0$ for $i=0,1$.
It follows that the right vertical arrow becomes an isomorphism
after taking direct limit in diagram (\ref{diagram: PT seq row 3 surjective case}).
Therefore we get the exactness of
$\BBP^0(K,C)\to \BBH^1(K,C')^D\to \BBH^1(K,C)$
and
$\BBP^1(K,C)_{\tors}\to \big(\BBH^0(K,C')_{\wedge}\big)^D\to \BBH^2(K,C)$ by diagram chasing.

\qed
\eitem


\begs\label{example: PT seq for tori mult type reductive group}
\item\label{example: PT seq for tori}
Let $P$ be a $K$-torus that extends to an $X_0$-torus $\CP$.
We consider the special case that $C=[0\to P]$ and $C'=P'[1]$.
By definition, $\BBH\ui(L,P)=0$ for $L=K$ or $K_v$ and hence the first two terms in diagram (\ref{diagram: 12-term PT for short complex}) vanish automatically.
The third term vanishies by \cref{lemma: list of groups II}.
Moreover, $\BBP^1(K,P)$ has finite exponent, so it is torsion.
Finally, $H^1(K,P')$ has finite exponent, thus the canonical map $H^1(K,P')\to H^1(K,P')_{\wedge}$ is an isomorphism.
The remaining $9$ terms in diagram (\ref{diagram: 12-term PT for short complex}) reads as
\[
\xm@C=20pt @R=12pt{
0\ar[r] & H^0(K,P)_{\wedge}\ar[r] & \BBP^0(K,P)_{\wedge}\ar[r] & H^2(K,P')^D
\ar@{->} `r/6pt[d] `/8pt[l] `^dl/8pt[lll] `^r/8pt[dl] [dll]\\
& H^1(K,P)\ar[r] & \BBP^1(K,P)\ar[r] & H^1(K,P')^D
\ar@{->} `r/6pt[d] `/8pt[l] `^dl/8pt[lll] `^r/9pt[dl] [dll]\\
& H^2(K,P)\ar[r] & \BBP^2(K,P)_{\tors}\ar[r] & \big(H^0(K,P')_{\wedge}\big)^D\ar[r] & 0
}
\]
which is the Poitou--Tate exact sequence for tori \cite{HSSz15}*{Theorem 2.9}.

\item\label{example: PT seq for mult type}
Let $M$ be a group of multiplicative type over $K$.
We may embed it into a \ses $0\to M\to T_1\to T_2\to 0$ with $T_1$ and $T_2$ being $K$-tori.
\inpart there is a quasi-isomorphism $M[1]\simeq C_M$ with $C_M=[T_1\stra{\rho} T_2]$.
In this case, $\cok\rho$ is trivial and we obtain a Poitou--Tate sequence for short complexes $C_M$ and thus for groups of multiplicative type.

\item\label{example: PT seq for reductive group}
Let $G$ be a connected reductive group over $K$.
Let $\Gsc$ be the universal covering of the derived subgroup $\Gss$ of $G$.
Let $\rho:\Gsc\to \Gss\to G$ be the composite.
Let $T$ be a maximal torus of $G$ and let $\Tsc\ce\rho\ui(T)$ be the inverse image of $T$ in $\Gsc$.
Thus $\Tsc$ is a \maxtorus of $\Gsc$.
Following \cite{Bor98}, we write $H^i_{\ab}(K,G)=\BBH^i(K,C)$ and $\P^i_{\ab}(K,G)=\P^i(K,C)$ with $C=[T^{\sconn}\to T]$ for the abelianized Galois cohomologies.
So the Poitou--Tate sequence (\ref{diagram: 12-term PT for short complex}) yields an \ttes for the abelianization of Galois cohomology of $G$ as follows
\[
\xm@C=20pt @R=12pt{
0\ar[r] & H\ui_{\ab}(K,G)\ar[r] & \BBP_{\ab}\ui(K,G)\ar[r] & \BBH^2(K,C')^D
\ar@{->} `r/5pt[d] `/8pt[l] `^dl/8pt[lll] `^r/8pt[dl][dll]\\
& H^0_{\ab}(K,G)_{\wedge}\ar[r] & \BBP^0_{\ab}(K,G)_{\wedge}\ar[r] & \BBH^1(K,C')^D
\ar@{->} `r/5pt[d] `/8pt[l] `^dl/8pt[lll] `^r/8pt[dl] [dll]\\
& H^1_{\ab}(K,G)\ar[r] & \BBP^1_{\ab}(K,G)\ar[r] & \BBH^0(K,C')^D
\ar@{->} `r/5pt[d] `/8pt[l] `^dl/8pt[lll] `^r/9pt[dl] [dll]\\
& H^2_{\ab}(K,G)\ar[r] & \BBP^2_{\ab}(K,G)_{\tors}\ar[r] & \big(\BBH\ui(K,C')_{\wedge}\big)^D\ar[r] & 0.
}
\]
Hopefully it will give a defect to strong approximation, which is analogous to the number field case \cite{Dem11AF}.
\eegs

Finally, we relate $\BBH^i(K,C')$ in
\cref{example: PT seq for tori mult type reductive group}\ref{example: PT seq for reductive group}
with a simpler cohomology group.
Recall that the algebraic fundamental group $\pialg{G}$
(see \cite{Bor98}*{\S1} and \cite{CT08-resolution-flasque}*{\S6} for more information)
of a connected reductive group $G$ is
$\pialg{G}\ce X_*(T)/\rho_*X_*(\Tsc)$ where $\rho_*:X_*(\Tsc)\to X_*(T)$ is induced by $\rho:\Tsc\to T$.

\begin{cor}\label{sequence: abel Kottwitz--Borovoi}
Let $G$ be a connected reductive group.
Let $C=[\Tsc\stra{\rho}T]$ be as above.
Let $G^*$ be the group of multiplicative type such that $X^*(G^*)=\pialg{G}$.
There is an \ttes
\[
0\to \Sha^1(C)\to \BBH^1(K,C)\to \BBP^1(K,C)\to H^1(K,G^*)^D.
\]
\end{cor}
\begin{proof}
Let $\Tss\ce T\cap \Gss$ and let $\Gtor\ce T/\Tss$.
Thus there is a \ses of short complexes
\be\label{sequence: ses of short complexes for red group}
0\to [(\Gtor)'\to 0]\to [T'\to (\Tsc)']\to [(\Tss)'\to (\Tsc)']\to 0.
\ee
By \cite{CT08-resolution-flasque}*{Proposition 6.4},
there is a \ses of abelian groups
\[
0\to (\ker\rho)(-1)\to \pialg{G}\to X_*(\Gtor)\to 0.
\]
Here $(\ker\rho)(-1)\ce\hom_{\Z}(X^*(\ker\rho),\QZ)$
is the module of characters of $(\ker\rho)'\ce \hom(\ker\rho,\QZ(2))$.
Thus there is an \ttes of groups of multiplicative type
\be\label{sequence: ses concerning fund group for red group}
0\to (\Gtor)'\to G^*\to (\ker\rho)'\to 0.
\ee

Since $\Tsc\to \Tss$ is an isogeny with kernel $\ker\rho$,
its dual isogeny $(\Tss)'\to (\Tsc)'$ has kernel $(\ker\rho)'$,
i.e. there is a quasi-isomorphic $[(\Tss)'\to (\Tsc)']\simeq (\ker\rho)'[1]$.
By definition there is an \ttes $X_*(\Tsc)\stra{\rho_*} X_*(T)\to \pialg{G}\to 0$,
so there is a corresponding \ttes $0\to G^*\to T'\to (\Tsc)'$ of groups of multiplicative type.
\inpart we obtain a morphism of short complexes $G^*[1]\to C'$.
Summing up,
there is a \cmdm of short complexes with exact rows obtained from (\ref{sequence: ses concerning fund group for red group}) and (\ref{sequence: ses of short complexes for red group}):
\[
\xm{
0\ar[r] & (\Gtor)'[1]\ar[r]\ar@{=}[d] & G^*[1]\ar[r]\ar[d] & (\ker\rho)'[1]\ar[r]\ar[d] & 0\\
0\ar[r] & (\Gtor)'[1]\ar[r] & C'\ar[r] & [(\Tss)'\to (\Tsc)']\ar[r] & 0.
}
\]
Since the right vertical arrow is a quasi-isomorphism,
so is the middle one after taking Galois hypercohomology and applying the $5$-lemma.
Thus $H^{i+1}(K,G^*)\simeq \BBH^i(K,C')$ and the desired sequence follows from \cref{example: PT seq for tori mult type reductive group}\ref{example: PT seq for reductive group}.
\end{proof}

\begin{rmk}
\cref{sequence: abel Kottwitz--Borovoi} gives an abelianized version of the Kottwitz--Borovoi sequence \cite{Bor98}*{Theorem 5.16} over $p$-adic function fields $K$.
Hopefully over such $K$ there is an \ttes $1\to \Sha^1(G)\to H^1(K,G)\to \BBP^1(K,G)\to H^1(K,G^*)^D$ of pointed sets for connected linear groups
which can be used to give an obstruction to \wa for homogenous spaces under some linear group with stabilizer $G$.
\end{rmk}

Univerit\'e Paris-Sud, Institut de Math\'ematique d'Orsay, B\^atiment 307, 91405 Orsay, France\\
\indent E-mail address: yisheng.tian@u-psud.fr

\end{document}